\documentclass[11pt,a4paper,reqno]{amsart}

\usepackage{amsmath,amsfonts,amssymb,amsthm}

\usepackage{bm}

\usepackage{verbatim}

\usepackage{color}

\usepackage[all]{xy}

\usepackage{hyperref}

\newlength{\wideitemsep}
\let\olditem\item
\renewcommand{\item}{\setlength{\itemsep}{\wideitemsep}\olditem}

\numberwithin{equation}{section}

\renewcommand{\geq}{\geqslant}
\renewcommand{\leq}{\leqslant}

\newcommand{\lra}{\longrightarrow}

\newcommand{\image}{\textnormal{im}\,}

\newcommand{\Coh}{\textnormal{Coh}}

\newcommand{\HHom}{\mathcal{H}om}

\newcommand{\OO}{\mathcal{O}}
\newcommand{\wh}{\widehat}

\newcommand{\cB}{\mathcal{B}}
\newcommand{\rk}{\mathrm{rank}}
\newcommand{\dimension}{\mathrm{dim}\,}

\makeatletter
\newtheorem*{rep@theorem}{\rep@title}
\newcommand{\newreptheorem}[2]{%
\newenvironment{rep#1}[1]{%
 \def\rep@title{#2 \ref{##1}}%
 \begin{rep@theorem}}%
 {\end{rep@theorem}}}
\makeatother

\newtheorem{theorem}{Theorem}[section]
\newtheorem{lemma}[theorem]{Lemma}
\newtheorem{sublemma}[theorem]{Sublemma}
\newtheorem{proposition}[theorem]{Proposition}
\newtheorem{corollary}[theorem]{Corollary}

\newreptheorem{theorem}{Theorem}
\newreptheorem{corollary}{Corollary}

\newtheorem*{theorem6-18}{Theorem 5.18}
\newtheorem*{theorem6-21}{Theorem 5.21}

\theoremstyle{definition}

\newtheorem{example}[theorem]{Example}

\theoremstyle{remark}
\newtheorem{remark}[theorem]{Remark}

\newcommand{\cA}{\mathcal{A}}
\newcommand{\cC}{\mathcal{C}}

\newcommand{\cE}{\mathcal{E}}
\newcommand{\cF}{\mathcal{F}}

\newcommand{\cO}{\mathcal{O}}
\newcommand{\cP}{\mathcal{P}}
\newcommand{\cQ}{\mathcal{Q}}

\newcommand{\cT}{\mathcal{T}}

\newcommand{\bR}{\mathbf{R}}
\newcommand{\bZ}{\mathbb{Z}}

\DeclareMathOperator{\coh}{Coh}

\DeclareMathOperator{\pic}{Pic}

\DeclareMathOperator{\Hom}{Hom}

\DeclareMathOperator{\amp}{Amp}
\DeclareMathOperator{\ns}{NS}
\DeclareMathOperator{\ch}{ch}

\DeclareMathOperator{\supp}{Supp}

\DeclareMathOperator{\rank}{rank}

\DeclareMathOperator{\chern}{ch}
\DeclareMathOperator{\td}{td}
\newcommand{\HH}{{\rm H}}

\newtheorem*{theorem4-13}{Theorem 4.12}
\newtheorem*{theorem4-26}{Theorem 4.26}

\begin{document}

\title{Preservation of semistability under Fourier-Mukai transforms}

\author[Jason Lo]{Jason Lo}
\address{Department of Mathematics \\
University of Illinois at Urbana-Champaign\\
1409 W Green St\\
Urbana IL 61801 \\
USA}
\email{jccl@illinois.edu}

\author{Ziyu Zhang}
\address{Department of Mathematical Sciences, University of Bath, Claverton Down, Bath BA2 7AY, United Kingdom}
\email{zz505@bath.ac.uk}


\begin{abstract}
For a trivial elliptic fibration $X=C \times S$ with $C$ an elliptic curve and $S$ a projective K3 surface of Picard rank $1$, we study how various notions of stability behave under the Fourier-Mukai autoequivalence $\Phi$ on  $D^b(X)$, where $\Phi$ is induced by the classical Fourier-Mukai  autoequivalence on $D^b(C)$.  We show that, under some restrictions on  Chern classes, Gieseker semistability on coherent sheaves is preserved under $\Phi$ when the polarisation is `fiber-like'.  Moreover, for more general choices of Chern classes, Gieseker semistability under a `fiber-like' polarisation corresponds to a notion of $\mu_\ast$-semistability defined by a  `slope-like' function $\mu_\ast$.
\end{abstract}


\maketitle

\section{Introduction}

This article is a continuation of the first author's study of elliptic fibrations.  In the predecessor to this article \cite{Lo11}, the first author studied t-structures that arise naturally on the derived category of coherent sheaves on an elliptic fibration.  In this article, we consider different notions of stability that can be paired with these t-structures, and describe how these notions of stability correspond to one another under Fourier-Mukai transforms.

In general, given an exact equivalence $\Phi : D^b(X) \overset{\thicksim}{\to} D^b(Y)$ between two derived categories and an object $E \in D^b(X)$ that is stable with respect a notion of stability $\sigma$ on $X$, we can ask: under what notion of stability $\sigma'$ on $Y$ (depending only on the geometry of $Y$) is $\Phi (E)$ stable?





The preservation of semistability for sheaves under Fourier-Mukai transforms has been studied by various authors before.    For example, in the case of sheaves, isomorphisms between moduli spaces of torsion-free sheaves were constructed by Yoshioka on K3 surfaces and Abelian surfaces \cite{YosPI, YosPII}, by Bernardara-Hein on elliptic surfaces \cite{BH} and by Bridgeland-Maciocia on elliptic threefolds \cite{BM04}.  On the other hand, isomorphisms between moduli spaces of torsion sheaves and torsion-free sheaves were constructed by Yoshioka on elliptic surfaces \cite{YosAS, YosPII} and Bridgeland on Abelian surfaces \cite{BriTh}.  And more recently, on elliptic threefolds, isomorphisms between moduli spaces of torsion sheaves were established by Diaconescu \cite{Dia}.


On elliptic surfaces, Yoshioka showed that, with respect to polarisations that approach the fiber direction, moduli spaces of semistable sheaves supported in dimension 1 are isomorphic to moduli spaces of semistable torsion-free (i.e.\ 2-dimensional) sheaves via a Fourier-Mukai transform \cite{YosAS, YosPII}.  In this article, we generalise Yoshioka's result to the threefold $X = C \times S$, where $C$ is an elliptic curve and $S$ a K3 surface of Picard rank 1.  The main reason for focusing on the product case in this article is to keep the computations involving Chern classes as tractable as possible.

We consider $X$ as a (trivial) elliptic fibration via the second projection $\pi : C \times S \to S$.  The classical Fourier-Mukai transform $D^b(C) \to D^b(C)$ on a smooth elliptic curve $C$, with the Poincar\'{e} line bundle as the kernel, pulls back via $\pi$ to a Fourier-Mukai transform $\Phi : D^b(X) \to D^b(X)$ on $X$.  Suppose the Picard group of $S$ is generated by the ample divisor class $H_S$.  Writing $D = \pi^\ast H_S$,  $H$ for the zero section of $\pi$, and $\omega = H+nD$, we have  the following two theorems:



\begin{theorem6-18}
For any fixed $\ch = (\ch_0, \ch_1, \ch_2, \ch_3) \in H^\ast_{\textrm{alg}}(X, \bZ)$ satsifying
\[
\ch_0 = 0, \ \ch_1 \cdot H \cdot D = 0, \ \ch_1 \cdot D^2 \neq 0, \ \ch_2 \cdot H = 0,
\]
there is an isomorphism between the following moduli spaces:
\begin{itemize}
\item[(a)] The moduli space of $\omega$-semistable sheaves $F$ on $X$ with $\ch (F) = \ch$;
\item[(b)] The moduli space of $\omega$-semistable sheaves $E$ for $n \gg 0$ on $X$ with $\ch (E) = \Phi^{H} (\ch)$.
\end{itemize}
\end{theorem6-18}
Here, $\Phi^{H}$ denotes the cohomological Fourier-Mukai transform induced by $\Phi$.  The moduli space in (a) parametrises torsion sheaves (supported in dimension 2), while the moduli space in (b) parametrises torsion-free sheaves (supported in dimension 3).

The constraints on the Chern character in Theorem \ref{0602-theorem1} can be relaxed, in which case  we have:

\begin{theorem6-21}
There is an equivalence of categories induced by $\Phi$
\begin{multline*}
\{ E \in \{\Coh^{\leq 0}\}^\uparrow : \ch_1(E) \neq 0, E \text{ is $\mu_\ast$-semistable}\} \\
 \overset{\thicksim}{\to} \{E \in \Phi (\{\Coh^{\leq 0}\}^\uparrow) : \ch_0(E) \neq 0, E \text{ is $\mu_\omega$-semistable for $n \gg 0$}\}.
\end{multline*}
\end{theorem6-21}
Here, the category $\{\Coh^{\leq 0}\}^\uparrow$ comprises coherent sheaves $E$ on $X$ satisfying the following property: for any closed point $s \in S$, the restriction $E|_s$ of $E$ to the fiber over $s$ has 0-dimensional support.  This category includes, for instance, sheaves supported on `spectral covers', which were used by Friedman-Morgan-Witten to construct stable sheaves on elliptic fibrations \cite{FMW}.  Also, the notion of $\mu_\ast$-semistability is given by the `slope-like' function $\mu_\ast$, which is defined using components of the first and the second Chern characters (see Section \ref{sec-slopelike} for the definition of a slope-like function).

In essence, Theorem \ref{0602-theorem1} says that $\mu_\omega$-semistability for $n \gg 0$ corresponds to the notion of $\mu_\ast$-semistability under the Fourier-Mukai transform $\Phi$.

In Section \ref{sec-prelim} of this article, we fix our notations and give reminders on the basics of Fourier-Mukai transforms and torsion classes.

In Section \ref{sec-slf}, we consider an arbitrary elliptic fibration $\pi : X \to S$ with a dual, and describe  a filtration for coherent sheaves on $X$ that replaces the relative Harder-Narasimhan filtration with respect to the morphism $\pi$  (Proposition \ref{pro6}).  The advantage of this filtration is that it makes no explicit reference to stability.  Then, we introduce the notion of a slope-like function.  A slope-like function yields a notion of semistability that has the Harder-Narasimhan property, and is closely related to the concept of a weak stability condition introduced by Toda \cite{Toda11-1} and the notion of a weak central charge defined by Brown-Shipman \cite{BS}.

In Section \ref{sec-FMTcohom}, we restrict our attention to the product threefold $X = C \times S$, where $C$ is an elliptic curve and $S$ a K3 surface of Picard rank 1. We describe the Fourier-Mukai transform $D^b(X) \overset{\thicksim}{\to} D^b(Y)$ that we will use in proving Theorems \ref{0602-theorem1} and \ref{0602-theorem2}, and describe the corresponding cohomological Fourier-Mukai transform.

In Section \ref{sec:pres}, we prove our main results, Theorems \ref{0602-theorem1} and \ref{0602-theorem2}.  And in the Appendix, we demonstrate that our approach recovers Yoshioka's result \cite[Theorem 3.15]{YosAS} in the case of the product $C \times T$, where $C$ is an elliptic curve and $T$ is an arbitrary smooth projective curve.

\noindent
\textbf{Acknowledgements.} We thank Arend Bayer for many helpful discussions during the entire span when this project is carried out. The first author would like to thank the Max Planck Institute for Mathematics in Bonn, for their hospitality during his stay from April through May 2014, during which this project was conceived.  He would also like to thank the National Center for Theoretical Sciences in Taipei, for their support during his visit from May to July 2015, when part of this work was completed. The second author would like to thank Alastair Craw for his support via the EPSRC grant EP/J019410/1.

\section{Preliminaries on elliptic fibrations and torsion classes}\label{sec-prelim}

In this section we collect some notions and results that we will use in later discussions.

\subsection{Elliptic fibrations}

In this paper, an \emph{elliptic fibration} is a flat morphism of smooth projective varieties $\pi: X \to S$ whose fibers are Gorenstein curves of arithmetic genus $1$. We are mainly interested in the case that $\dim X = 3$, although we will also consider the case $\dim X=2$ in the final section.

We mention a few examples of elliptic threefolds. The simplest example is the trivial fibration, namely, $X=C \times S$ for some elliptic curve $C$ and surface $S$, and $\pi$ is the projection to the second factor. This is the example that we will focus on in this paper. Indeed, as we will study derived categories of the elliptic threefold $X$, which has abundant structures, it helps to keep the structure of the elliptic fibration itself simple. In fact, even in the case of trivial fibrations, we can already discover lots of interesting features of the derived category of $X$.

We will also use the following notations throughout our discussion.

\begin{itemize}
\item For any smooth projective variety $X$, we will write $D(X)=D^b(\Coh (X))$ to denote the bounded derived category of coherent sheaves on $X$ unless otherwise stated.
\item For any smooth projective variety $X$ and any integer $d \geq 0$, we will write $\Coh^{\leq d}(X)$ to denote the category of coherent sheaves on $X$ supported in dimension at most $d$, write $\Coh^{=d}(X)$ to denote the category of pure $d$-dimensional coherent sheaves, and write $\Coh^{\geq d}(X)$ to denote the category of coherent sheaves with no nonzero subsheaves supported in dimension $d-1$ or lower.
\item For any smooth projective variety $X$ of dimension $d$, we will sometimes write $\mathcal F_X$ instead of $\Coh^{=d}(X)$ to denote the category of torsion-free sheaves on $X$.
\item Assume $\pi: X \to S$ is an elliptic fibration as above. For any closed point $s \in S$, we will write $\iota_s$ to denote the closed immersion of the fiber $X_s \to X$ of $\pi$ over  $s$.  When $E$ is a coherent sheaf on $X$, we will write $E|_s$ to denote the underived restriction $\iota_s^\ast E$, and write $E|_s^L$ to denote the derived restriction $L\iota_s^\ast E$.
\item If a coherent sheaf $E$ on $X$ is supported on a finite number of fibers of $\pi$, then we will refer to it as a \emph{fiber sheaf}.
\item We write $\Coh (\pi)_{\leq d}$ to denote the category of coherent sheaves $E$ on $X$ such that the dimension of $\pi (\supp(E))$ is at most $d$.
\item We write $\{\Coh^{\leq 0}(X_s)\}^\uparrow$ to denote the category of coherent sheaves $E$ on $X$ such that the restriction $E|_s$ to the fiber over any closed point $s \in S$ is supported in dimension 0. When $X$ is understood from the context, we also write $\{\Coh^{\leq 0}\}^\uparrow$ for simplicity.
\end{itemize}

\subsection{Fourier-Mukai transforms}

Our main tool to study the derived category of $X$ is (relative) Fourier-Mukai transforms. The theory of Fourier-Mukai transforms on elliptic fibrations has been developed by many authors. One of the general results that we found most inspiring is the following theorem due to Bridgeland and Maciocia. Notice that their original statement is slightly more general than the following statement.

\begin{theorem}
\cite[Theorem 1.2]{BM04}
Let $\pi: X \to S$ be an elliptic fibration and $\ell$ a polarisation on $X$. Let $Y$ be an irreducible component of the relative moduli space of $\ell$-stable sheaves on fibers of $\pi$. Assume that $Y$ is fine and of the same dimension as $X$, and that the morphism $\hat{\pi}: Y \to S$ is equidimensional. Then $\hat{\pi}: Y \to S$ is also an elliptic fibration. Moreover, the following integral functor induced by the universal sheaf $\cP$ on $X \times Y$ is a Fourier-Mukai transform:
\begin{align}
\label{eqn:FM-def}
\Phi=\Phi^{\cP}_{Y \to X}: \quad D^b(Y) &\to D^b(X); \\
\cE &\mapsto \bR\pi_{X,\ast}(\cP \otimes \pi^{\ast}_Y(\cE)), \nonumber
\end{align}
where $\pi_X$ and $\pi_Y$ are projections from $X \times Y$ to the two factors.
\end{theorem}

In the above theorem, the two elliptic fibrations $\pi: X \to S$ and $\hat{\pi}: Y \to S$ are called \emph{dual fibrations} of each other. The Fourier-Mukai transform $\Phi=\Phi^{\cP}_{Y \to X}$ has a quasi-inverse, which is given by
$$ \Psi=\Phi^{\cQ}_{X \to Y}: D^b(X) \to D^b(Y), $$
where $\cQ$ is the object
$$ \cQ = \bR\HHom_{\cO_{X \times Y}}(\cP, \pi_X^\ast\omega_X)[\dim X]. $$
Notice that if the Fourier-Mukai kernel $\cP$ is a line bundle, then $\cQ$ is also a (shifted) line bundle.

With some extra assumption on the elliptic fibration $\pi: X \to S$, the above result can be made more explicit. We say an elliptic fibration $\pi: X \to S$ is a Weierstrass fibration if the fibers of $\pi$ are integral and there exists a section $\sigma: S \to X$ of $\pi$ whose image does not contain any singular point of the fibers. When $\pi: X \to S$ is a Weierstrass fibration, we can make a particular choice of $Y$ and $\cP$ so that $\Phi$ becomes an automorphism of $D^b(X)$.

\begin{theorem}
\cite[Theorem 6.18]{FMNT}
Let $\pi: X \to S$ be a Weierstrass fibration and $\cP$ be the relative Poincar\'e sheaf on $X \times X$, then the integral functor $\Phi = \Phi^{\cP}_{X \to X}$ defined as in \eqref{eqn:FM-def} is a Fourier-Mukai transform.
\end{theorem}

The advantage of having a Fourier-Mukai type automorphism $\Phi$ on $D^b(X)$ is that we can use it to study how various subcategories change under $\Phi$. It would be nice to see that different structures on $D^b(X)$ are related by this automorphism. The case that we are mostly interested in, namely the trivial fibration, is clearly a Weierstrass fibration. Understanding the action of the automorphism $\Phi$ induced by the relative Poincar\'e bundle $\cP$ on $D^b(X)$ will be our major goal in this paper.

A key property that often comes up in the study of Fourier-Mukai transforms is the so-called WIT property. Let $\Phi: D^b(Y) \to D^b(X)$ be any Fourier-Mukai transform. A complex  $E \in D^b(Y)$ is said to be $\Phi$-\emph{WIT$_i$} if $\Phi(E) = F[-i] \in D^b(X)$ for some $F \in \coh(X)$ and $i \in \bZ$; in other words, the image of $E$ under the Fourier-Mukai transform $\Phi$ is a sheaf sitting at  degree $i$.  When $E \in D^b(Y)$ is a $\Phi$-WIT$_i$ complex, we write $\wh{E}$ to denote the coherent sheaf $F=\Phi (E) [i]$, i.e.\ $\Phi (E) \cong \wh{E} [-i]$.  As every object in $D^b(Y)$ is a finite extension of coherent sheaves and their shifts, understanding lots of examples of sheaves in $\coh(Y)$ satisfying WIT properties in various degrees can usually help us determine the image of a subcategory of $D^b(Y)$ under $\Phi$.

We introduce the following extra notations that will be used later.
\begin{itemize}
\item For any object $E \in D^b(Y)$, we write $\Phi^i(E)$ for the cohomology sheaf $H^i(\Phi(E))$ with respect to the standard t-structure.
\item We write $W_{i,X}$ to denote the category of $\Psi$-WIT$_i$ sheaves.
\item We write $W_{i,X}'$ to denote the category of coherent sheaves $E$ on $X$ such that, for a general closed point $s \in S$, the image of $E|_s$ under $\Psi_s$ (the base change of $\Psi$ under the closed immersion $\{s\}\hookrightarrow S$) is a sheaf sitting at degree $i$.
\item We write $\mathfrak{T}_X$ to denote the extension closure $\langle W_{0,X}', \Phi (\{\Coh^{\leq 0}(Y_s)\}^\uparrow \rangle$.
\item We write $\mathfrak{T}'_X$ to denote the extension closure $\langle W_{0,X}', \Phi (W_{0,Y} \cap \Coh^{\leq n-1}(Y) )\rangle$ where $n$ equals $\dim X$ (which equal $\dim Y$).
\end{itemize}

\subsection{Torsion classes}

Another important notion that shows up quite often in our discussion is a torsion pair. For any abelian category $\cA$, a \emph{torsion pair} $(\cT, \cF)$ in $\cA$ is a pair of full subcategories of $\cA$, satisfying the following two conditions
\begin{enumerate}
\item Every object $E \in \cA$ fits in a short exact sequence $ 0 \to T \to E \to F \to 0$ in $\cA$ where $T \in \cT$ and $F \in \cF$;
\item For any $T \in \cT$ and $F \in \cF$, we have $\Hom(T,F)=0$.
\end{enumerate}
For such a torsion pair $(\cT, \cF)$ in $\cA$, we refer to $\cT$ as the torsion class of the pair, and $\cF$ as the torsion-free class of the pair. We say a full subcategory $\cC$ of $\cA$ is a \emph{torsion class} if it is the torsion class of a torsion pair in $\cA$. We will need the following useful criterion to recognize torsion classes:

\begin{lemma}
\label{lem:torsion-pair}
\cite[Lemma 1.1.3]{Pol}
Let $\cA$ be a noetherian abelian category. If a full subcategory $\cT$ is closed under quotients and extensions, then $\cT$ is a torsion class.
\end{lemma}

Under the same assumption, we write
    \[
      \cC^\circ := \{ E \in \cA \mid \Hom_{\cA}(C,E)=0 \text{ for all } C \in \cC\}.
    \]
Then the torsion pair in Lemma \ref{lem:torsion-pair} is given by $(\cC, \cC^\circ)$.

Whenever we have a torsion pair $(\mathcal T, \mathcal F)$ in an abelian category $\cA$, there is a corresponding t-structure $(D^{\leq 0}, D^{\geq 0})$ on the derived category $D(\mathcal A)$ given by
    \begin{align*}
      D^{\leq 0} &:= \{ E \in D(\mathcal A) \mid H^0(E) \in \mathcal T, H^i(E)=0 \text{ for all $i>0$}\},\\
      D^{\geq 0} &:= \{ E \in D(\mathcal A) \mid H^{-1}(E) \in \mathcal F, H^i(E)=0 \text{ for all $i<-1$ }\}.
    \end{align*}
The heart of this t-structure is given by
$$\cA^\# = \{E \in D(\mathcal A) \mid H^{-1}(E) \in \mathcal F, H^0(E) \in \mathcal T, H^i(E)=0 \text{ for all $i\neq -1,0$}\},$$
It is called the \emph{tilted heart} with respect to the torsion pair $(\cT,\cF)$.

\section{Slope-like functions}\label{sec-slf}

\subsection{An alternative to the relative HN filtration}\label{subsec-alt}

In this section, we consider a pair of dual elliptic fibrations $\pi : X \to S$ and $\wh{\pi} : Y \to S$, and introduce a filtration for coherent sheaves on $X$ that replaces the use of relative HN filtrations with respect to $\pi$.  We assume that, for any closed point $x$ on a general (smooth) fiber of $\pi$, the sheaf $\Psi (\OO_x)$ is a line bundle of degree zero on the fiber of $\wh{\pi}$ over $\pi (x)$, and similarly for $\Phi$.




In \cite{Lo11}, we defined $\mathfrak T_X$ as
\[
  \mathfrak T_X := \langle W_{0,X}', \Phi (\{\Coh^{\leq 0}(Y_s)\}^\uparrow \rangle.
\]
We now modify the definition of $\mathfrak T_X$ a little, and set (where $d := \mathrm{dim}\, X$)
\begin{equation}\label{eq48}
  \mathfrak T_X' := \langle W_{0,X}', \Phi (W_{0,Y} \cap \Coh^{\leq d-1}(Y) )\rangle.
\end{equation}

\begin{remark}\label{remark1-1}
Note that, the torsion sheaves on $Y$ (i.e.\ objects in $\Coh^{\leq d-1}(Y)$ when $Y$ is smooth projective) of dimension $d$ are precisely the coherent sheaves $E$ on $Y$ such that $F|_s$ is a 0-dimensional sheaf for a general closed point $s \in S$.
\end{remark}

\begin{lemma}\label{lemma16}
We have
\begin{itemize}
\item[(i)] $\mathfrak T_X'$ is a torsion class in $\Coh (X)$.
\item[(ii)] When $X$ is an elliptic surface, the categories $\mathfrak T_X$ and $\mathfrak T_X'$ coincide.
\item[(iii)] $\mathfrak T_X'$ is also the category of all coherent sheaves $E$ on $X$ such that, for a general closed point $s \in S$, all the HN factors of $E|_s$ have $\mu \geq 0$.
\end{itemize}
\end{lemma}

\begin{proof}
(i) This follows from the same argument as in the proof of \cite[Lemma 3.19(iii)]{Lo11}, by replacing the category $\{\Coh^{\leq 0}(Y_s)\}^\uparrow$ with $W_{0,Y} \cap \Coh^{\leq 2}(Y)$.

(ii) Suppose $X,Y$ are dual elliptic surfaces.  By \cite[Lemma 3.15]{Lo11},  any $F \in W_{0,Y} \cap \Coh^{\leq 1}(Y)$ fits in a short exact sequence in $\Coh (Y)$ of the form
\[
0 \to F' \to F \to F'' \to 0
\]
where $F' \in \{\Coh^{\leq 0}(Y_s)\}^\uparrow$ and $F''$ is a $\Phi$-WIT$_0$ fiber sheaf, which implies $\mathfrak T_X' \subseteq \mathfrak T_X$.  That $\mathfrak T_X \subseteq \mathfrak T_X'$  follows from $\{\Coh^{\leq 0}(Y_s)\}^\uparrow \subseteq W_{0,Y}$ \cite[Remark 3.14]{Lo11}, thus proving (ii).

(iii) Given any object $F \in \mathfrak T_X'$, we know from \cite[Corollary 3.29]{FMNT} that, for a general closed point $s \in S$, all the HN factors of $F|_s$ have $\mu \geq 0$.

On the other hand, take any coherent sheaf $F$ such that all the HN factors of $F|_s$ have $\mu \geq 0$ for a general closed point $s \in S$.  By the existence of the relative HN filtration of $F$ with respect to $\pi$ \cite[Theorem 2.3.2]{HL10}, we can find a (possibly zero) subsheaf $G \subseteq F$ satisfying the following properties:
\begin{itemize}
\item[(a)] all the HN factors of $G|_s$ have $\mu >0$ for a general closed point $s \in S$ (and so $G \in W_{0,X}'$);
\item[(b)] $F/G$ is torsion-free;
\item[(c)] $(F/G)|_s$ is $\mu$-semistable with $\mu=0$ for a general closed point $s \in S$.
\end{itemize}
Since $F/G$ lies in $W_{1,X}'$ and is torsion-free, it lies in $W_{1,X}$ by \cite[Lemma 3.18(ii)]{Lo11}.  Property (c) now implies that  $(\wh{F/G})|_s$ is 0-dimensional for a general closed point $s \in S$  \cite[Corollary 3.29]{FMNT}, and so  $\wh{F/G} \in \Coh^{\leq d-1}(Y)$, implying $F/G \in \mathfrak T_X'$.  This completes the proof of (iii).
\end{proof}

\begin{proposition}\label{pro6}
Let $\pi : X \to S$ and $\hat{\pi} : Y \to S$ be dual elliptic fibrations of dimension $d$ such that the Fourier-Mukai kernels of $\Phi$ and $\Psi$ are both coherent sheaves. Then
\begin{itemize}
\item[(a)] Any coherent sheaf $E$  on $X$ has a filtration
\begin{equation}\label{eq44}
  G_0 \subseteq G_1 \subseteq G_2 := E
\end{equation}
  in $\Coh (X)$ where the subfactors satisfy:
   \begin{itemize}
   \item[(i)] $G_0 \in W_{0,X}'$;
   \item[(ii)] $G_1/G_0 \in \mathcal F_X \cap \Phi (W_{0,Y} \cap \Coh^{\leq d-1}(Y))$, and so $G_1 \in \mathfrak T_X'$;
   \item[(iii)] $G_2/G_1  \in (\mathfrak T_X')^\circ$.  In particular, $G_2/G_1$ lies in $W_{1,X}$ and  is torsion-free.
   \end{itemize}
\item[(b)] Suppose, in addition, that $\pi, \hat{\pi}$ are dual elliptic fibrations satisfying the condition laid out at the start of Section \ref{subsec-alt}. With notation as in part (a), we have:
\begin{itemize}
\item[(i)] For a general closed point $s \in S$, $\mu >0$ for any HN factor of $G_0|_s$;
\item[(ii)] For a general closed point $s \in S$, $\mu =0$ for any HN factor of $(G_1/G_0)|_s$;
\item[(iii)] For a general closed point $s \in S$, $\mu <0$ for any HN factor of $(G_2/G_1)|_s$.
\end{itemize}
\end{itemize}
\end{proposition}


\begin{proof}
Take any $E \in \Coh (X)$, and set $G_2 := E$.  First, consider the exact sequence in $\Coh (X)$
\[
0 \to G_1 \to G_2 \to G_2/G_1 \to 0
\]
where $G_1 \in \mathfrak T_X'$ and $G_2/G_1 \in (\mathfrak T_X')^\circ$.
Then consider the exact sequence in $\Coh (X)$
\[
 0 \to G_0 \to G_1 \to G_1/G_0 \to 0
\]
where $G_0 \in W_{0,X}'$ and $G_1/G_0 \in (W_{0,X}')^\circ$.

We now verify that $G_1/G_0 \in \Phi (W_{0,Y} \cap \Coh^{\leq d-1}(Y))$.  Write $F := G_1/G_0$; note that  $F \in \mathfrak T_X' \cap (W_{0,X}')^\circ$.  Then $F$ is torsion-free and $\Psi$-WIT$_1$, since $W_{0,X}'$ contains all the $\Psi$-WIT$_0$ sheaves and torsion sheaves.  Also, from the definition of $\mathfrak T_X'$, we have that for a general closed point $s \in S$, all the HN factors of $F|_s$ have $\mu \geq 0$.  That $F \in (W_{0,X}')^\circ$ then  implies that for a general closed point $s \in S$, the restriction $F|_s$ is in fact $\mu$-semistable with $\mu=0$.  Hence $\wh{F}|_s$ is 0-dimensional for a general closed point $s \in S$ \cite[Corollary 3.29]{FMNT}, i.e.\ $\wh{F} \in \Coh^{\leq d-1}(Y)$.  Thus $F \in \Phi (W_{0,Y} \cap \Coh^{\leq d-1}(Y))$ as claimed.  That $G_2/G_1$ lies in $W_{1,X}$ and is torsion-free follows from $\mathfrak T_X'$ containing $W_{0,X}$ and all the torsion sheaves.  This completes the proof of part (a).

Parts (b)(i) and (b)(ii) follow immediately from \cite[Corollary 3.29]{FMNT}.  To see part (b)(iii), note that we have $G_2/G_1 \in (\mathfrak T_X')^\circ$ from part (a)(iii).  Lemma \ref{lemma16}(iii) then forces all the HN factors of $(G_2/G_1)|_s$ to have $\mu <0$ for a general closed point $s \in S$, giving us (b)(iii).
%
%
%
%
\end{proof}

\subsection{HN filtrations in abelian categories with slope-like functions}\label{sec-slopelike}

Give a noetherian abelian category $\cA$, we say that $\cA$ has a slope-like function $\mu$ if we have a pair of group homomorphisms $C_0 : K(\cA) \to \mathbb{Z}$, $C_1  : K(\cA) \to \mathbb{R}$ satisfying
\begin{itemize}
\item $C_0 (E) \geq 0$ for any $E \in \cA$;
\item if an object $E \in \cA$ satisfies $C_0 (E)=0$, then  $C_1 (E) \geq 0$.
\end{itemize}
We then define the function $\mu : K(\cA) \to \mathbb{R} \cup \{+\infty\}$ by setting
\begin{equation}\label{paper13-def1}
  \mu (E) := \begin{cases} \frac{C_1(E)}{C_0(E)} &\text{ if $C_0(E) \neq 0$} \\
    +\infty &\text{ if $C_0 (E)=0$},
    \end{cases}
\end{equation}
and refer to $\mu$ as a slope-like function.  This  yields a notion of semistability akin to the usual notion of slope semistability for coherent sheaves on smooth projective varieties.
We can also define the following subcategories of $\cA$,
\begin{align}
  \cB_0 &:= \{ E \in \cA : C_0(E)=0 \}, \notag\\
  \cB_{01} &:= \{E \in \cA : C_0(E)=0=C_1(E)\}. \label{sd-eqC01}
\end{align}
 It is easy to see that $\cB_{01} \subseteq \cB_0$ and both are Serre subcategories of $\coh (X)$ (and, in particular, are torsion classes  in $\coh (X)$  by Lemma \ref{lem:torsion-pair}).

Given a slope-like function $\mu$ on a noetherian abelian category $\cA$, we define an object $E \in \cA$ to be $\mu$-semistable if, for any  $0 \neq F \subsetneq E$ in $\cA$, we have $\mu (F) \leq \mu (E)$.  Note that a $\mu$-semistable object $E \in \cA$ necessarily lies in $(\cB_0)^\circ$, and so satisfies $C_0(E)>0$ if $E \neq 0$.

We define a nonzero object $E \in \cA$ to be $\mu$-stable if, for any $0 \neq F \subsetneq E$ in $\cA$ with $0<C_0(F)<C_0(E)$, we have $\mu (F)<\mu (E)$.  It can be checked easily that $\mu$-stable objects are $\mu$-semistable.

On the other hand, for any object $E \in K(\cA)$, we can associated to it a polynomial
\[
  P(E)(m) := C_0(E)m + C_1(E),
\]
and an associated `reduced' polynomial
\[
  p(E)(m) := P(E)(m)/\alpha(E)
\]
where $\alpha (E)$ is the leading coefficient of $P(E)(m)$.  There is a natural ordering $\leq$ on the reduced  polynomials (see \cite[Section 1.2]{HL}): for objects $F, E \in \cA$, we write $p(F)\leq p(E)$ if $p(F)(m) \leq p(E)(m)$ for $m \gg 0$.

We then define an object $E \in (\mathcal B_0)^\circ$ to be $p$-semistable if, for any $0 \neq F \subsetneq E$ in $\cA$, we have $p(F) \leq p(E)$, and define it to be $p$-stable if we always have strict inequality.

Note that, when $E \in \cA$ satisfies $C_0(E) \neq 0$ (e.g.\ when $0\neq E \in (\cB_0)^\circ$), we have
\begin{equation}
p(E)(m) = m + \mu (E).
\end{equation}
It is then easy to see the following:
\begin{enumerate}
\item For objects in $(\cB_0)^\circ$, being $\mu$-semistable is equivalent to being $p$-semistable.
\item If an object $E \in (\cB_0)^\circ$ is $p$-stable, then it is $\mu$-stable.
\item If an object $E \in \cA$ is $\mu$-stable, then it lies in $(\cB_0)^\circ$ and is $p$-semistable.  An example to keep in mind for this case is the following: let $\cA = \coh (X)$ where $X$ is a smooth projective surface, let $C_0$ be the rank function and $C_1$ be the degree function (with respect to some polarisation on $X$) on $K(\cA)$.  For any 0-dimensional subscheme $Z \subset X$, if we write $I_Z$ for the corresponding ideal sheaf then $C_0(I_Z)=C_0(\OO_X)$ and $C_1(I_Z)=C_1(\OO_X)$.  Hence $\OO_X$ is $\mu$-stable but only $p$-semistable.
\end{enumerate}

Our notion of slope-like functions here is similar to, but not the same as, the notion of a stability structure on an abelian category as in Gorodentsev-Kuleshov-Rudakov \cite{GKR}.   For instance:
\begin{itemize}
\item In order to obtain Harder-Narasimhan filtrations, Gorodentsev-Kuleshov-Rudakov required that the abelian category $\cA$ be weakly artinian and weakly noetherian (see \cite[Theorem 2.2]{GKR}); in our definition of a slope-like function, we require our abelian category $\cA$ to be noetherian.
\item Given an abelian category $\cA$, we are able to define a slope-like function provided we have additive functions $C_0, C_1$ on $K(\cA)$ satisfying the properties \eqref{sd-eqC01}. Gorodentsev-Kuleshov-Rudakov has a similar formulation in \cite[Definition 2.3]{GKR}, but they require $C_0, C_1$ to be linearly independent, whereas we do not require it here.  The result is that a slope-like function may well give rise to very degenerate Harder-Narasimhan filtrations.  For instance, if $\cA$ is a noetherian abelian category with $C_0, C_1$ satisfying \eqref{sd-eqC01} and $C_0=C_1$, then for any $E \in \cA$ with $C_0(E) \neq 0$, we have $\mu (E)=1$.  That is, any object in $(\cB_0)^\circ$ is $\mu$-semistable, and so its HN filtration has exactly one term.
\end{itemize}


\begin{proposition}\label{pro5}
If a noetherian abelian category $\cA$ has a slope-like function $\mu$, then every object in $\cA$ has a unique Harder-Narasimhan filtration of the form
\begin{equation}\label{eq47}
 E_{01} \subseteq E_0 \subseteq E_1 \subseteq E_2 \subseteq \cdots \subseteq E_m = E
\end{equation}
for some positive integer $m$ where
\begin{itemize}
\item $E_{01} \in \cB_{01}$,
\item $E_0/E_{01} \in \cB_0 \cap (\cB_{01})^\circ$,
\item for $i \geq 1$, each $E_i/E_{i-1}$ lies in $(\cB_0)^\circ$, is $\mu$-semistable, and
\[
  \mu (E_1/E_0) > \mu (E_2/E_1) > \cdots > \mu (E_m/E_{m-1}).
\]
\end{itemize}
\end{proposition}

When $E$ lies in $(\cB_0)^\circ$ and $E_i$ are  as in \eqref{eq47}, we define $\mu_{max} (E) := \mu (E_1/E_0)$ and $\mu_{min} (E) := \mu (E_m/E_{m-1})$.

Towards proving Proposition \ref{pro5}, we first prove:

\begin{lemma}\label{lemma17}
Suppose $E$ is an object in $(\mathcal B_0)^\circ$.  Then there is a subobject $F \subseteq E$ such that we have $\mu (G) \leq \mu (F)$ for any subobject $G$ of $E$, with equality only if $G \subseteq F$.  Moreover, $F$ is uniquely determined and is $\mu$-semistable.
\end{lemma}

Proposition \ref{pro5} has appeared in various forms including slope stability for coherent sheaves (see \cite[Theorem 1.6.7]{HL}) and tilt stability for complexes (see \cite[Lemma 3.2.4]{BMT14}).  Although the proof of \cite[Theorem 1.3.4]{HL} applies essentially without change, we still include the proof here to emphasise the point, once and for all, that the result holds for any noetherian abelian catgory with a slope-like function.

\begin{proof}[Proof of Lemma \ref{lemma17}]
We can assume $E \neq 0$.  Proceeding as in the proof of \cite[Lemma 1.3.5]{HL}, we consider an ordering $\leq$ on the nonzero subobjects of $E$, where we define $F_1 \leq F_2$ whenever
\[
  F_1 \subseteq F_2 \text{ and } p (F_1) \leq p (F_2).
\]
(Recall that, for objects $A, B$ with nonzero $C_0$, the condition $p(A) \leq p(B)$ is equivalent to $\mu(A) \leq \mu (B)$.) Since $\cA$ is a noetherian abelian category, every ascending chain of subobjects of $E$ with respect to $\leq$ in $\cA$ becomes stationary; among the maximal objects with respect to $\leq$, fix any one $F$ with minimal $C_0$ (here, we use the assumption that the codomain of the function $C_0$ is $\mathbb{Z}$).

Suppose that there exists an object $G \subseteq E$ with $p (G) > p (F)$.  We first show that we can then assume $G \subseteq F$: consider the short exact sequence in $\cA$
\begin{equation}\label{eq46}
0 \to F \cap G \to F \oplus G \to F + G \to 0.
\end{equation}
 Note that, since $E$ is nonzero and lies in $(\cB_0)^\circ$, we have $C_0(E) \neq 0$.  Since $F, G$ are also nonzero and lie in $(\cB_0)^\circ$ (by virtue of being subobjects of $E$), we also have $C_0(F), C_0(G) \neq 0$, and so $p(F), p(G), p(F+G)$ are all polynomials of degree 1. Now, if $F \cap G = 0$, then $F+G \cong F \oplus G$, and from the short exact sequence in $\cA$
\[
0 \to F \to F \oplus G \to G \to 0
\]
together with $p (F)< p (G)$, we obtain $p (F) < p (F\oplus G)= p (F+G)$ while $F \subsetneq F+G$ and the maximality of $F$ is contradicted.  So we can assume $F \cap G \neq 0$.  Then $C_0(F \cap G) >0$, and from \eqref{eq46} we obtain:
\begin{equation}\label{eq45}
C_0 (F \cap G)  (p(G) - p(F \cap G) ) = C_0 (F+G) (p(F+G) - p(F)) + (C_0(G)-C_0(F \cap G))(p(F)-p(G)).
\end{equation}
The maximality of $F$ implies $p (F) \geq p (F+G)$.  This, together with the assumption $p(F)<p(G)$, implies that the right-hand side of \eqref{eq45} is non-positive, and so $p(F \cap G) \geq p(G) > p(F)$.  That is, we have produced a nonzero subobject $F \cap G$ of $F$ satisfying $p (F \cap G) > p(F)$.  Therefore, replacing $G$ by $F \cap G$, we will assume that $G \subseteq F$ from now on.  In addition, we will assume that $G$ is $\leq$-maximal among all the subobjects $G_0$ of $F$ such that $p(G_0) > p (F)$.

Next, we let $G'$ be any subobject of $E$ containing $G$ and is $\leq$-maximal.  Then $p(F) < p (G) \leq p(G')$.  By the maximality of $G'$ and $F$, we now must have $G' \nsubseteq F$.  For, if $G' \subseteq F$, then $C_0 (G') \leq C_0 (F)$.  If equality holds, then $C_0 (F/G')=0$, which implies that $C_1 (F/G') \geq 0$, and so $p (G') \leq p (F)$, a contradiction.  On the other hand, if $C_0 (G') < C_0 (F)$, then the minimality assumption on $C_0 (F)$ is contradicted.  Hence $G' \subseteq F$ cannot hold.  As a result, we have $F \subsetneq F + G'$, and the $\leq$-maximality of $F$ implies that $p(F) > p(F+G')$.

Overall, we now have $p(F)>p(F+G')$ and $p(F)< p(G')$. Replacing $G$ by $G'$ in the short exact sequence \eqref{eq46} and in \eqref{eq45}, we obtain that the right-hand side of \eqref{eq45} is strictly negative, and so, since $C_0 (F\cap G') >0$ (since $F, G'$ both contain $G$ as a subobject), we have $p(F \cap G')>p(G')\geq p(G)$.  However, the object $F \cap G'$ then contradicts the maximality of $G$.

Therefore, we have proved that for any nonzero $G \subseteq E$, we have $p (G) \leq p (F)$, i.e.\ $\mu (G) \leq \mu (F)$.

Next, we show that if $G \subseteq E$ is such that $p (F) = p (G)$, then it must follow that $G \subseteq F$: for any such $G$, if we have $G \nsubseteq F$, then $F \subsetneq F +G$, and by the $\leq$-maximality of $F$, we must have $p (F) > p (F+G)$.  Now, $F \cap G$ cannot be zero, or else $F+G \cong F \oplus G$, and then we would have $F \subsetneq F+G \subseteq E$ while $p (F) = p (F+G)$, contradicting the $\leq$-maximality of $F$. Thus $F \cap G \neq 0$, and the inequality $p (G) = p (F) > p (F+G)$ implies, from \eqref{eq46},
\[
  p (F\cap G) > p (F\oplus G) = p (F),
\]
contradicting the conclusion in the previous paragraph.  Therefore, any $G \subseteq E$ with $p (F) = p (G)$ must satisfy  $G \subseteq F$.

The rest of the lemma is clear.
\end{proof}

The standard argument (see \cite[Proposition 1.2.7]{HL}, for instance) gives us the following property of $\mu$-semistable objects in $\cA$:

\begin{lemma}\label{lemma18}
Given $\mu$-semistable objects  $F, G \in (\mathcal B_0)^\circ$ with $\mu (F)> \mu (G)$, we have $$\Hom_{\cA}(F,G)=0.$$
\end{lemma}

\begin{proof}[Proof of Proposition \ref{pro5}]
Take any $E \in \cA$.  Set $E_0$ to be the maximal subobject of $E$ in the torsion class $\cB_0$, and set $E_{01}$ to be the maximal subobject of $E_0$ in the torsion class $\cB_{01}$.  Then $E_0/E_{01} \in (\cB_0) \cap (\cB_{01})^\circ$, and $E/E_0 \in (\cB_0)^\circ$.  The rest of the proof uses Lemma \ref{lemma18} and follows the same argument as that for Gieseker stability for coherent sheaves - see \cite[Theorem 1.3.4]{HL}.
\end{proof}



\begin{example}\label{paper13-eg1}
Let $X$ be a smooth projective variety of dimension $n$, and let $\omega$ be a fixed ample divisor class on $X$.  If we take $\mathcal A := \Coh (X)$ and
\begin{align*}
  C_0 (-) &:= \rank (-), \\
  C_1 (-) &:= c_1(-) \cdot \omega^{n-1},
\end{align*}
then the resulting slope-like function $\mu := C_1(-)/C_0(-)$ coincides with the usual slope function $\mu_\omega (-)$ that gives rise to the notion of slope semistability  for coherent sheaves on $X$.

More generally, if $n := \dimension X$ and $0 \leq d \leq n$ is any integer, then the functions
\begin{align*}
  C_0(-) &:= \ch_{n-d}(-)\cdot \omega^d \\
  C_1(-) &= \ch_{n-d+1}(-)\cdot \omega^{d-1}
\end{align*}
give a slope-like function on the noetherian abelian subcategory $\mathcal A := \Coh^{\leq d}(X)$ of $\Coh (X)$.
\end{example}

\begin{example}\label{paper13-eg2}
Let $\pi : X \to S$ be any flat morphism of relative dimension 1, where $X, S$ are both smooth projective varieties.  Let $f$ denote the fiber class for the morphism $\pi$.  Then we can take $\mathcal A := \Coh (X)$, and
\begin{align*}
  C_0 (-) &:= \rank (-),\\
  C_1 (-) &:= c_1(-)\cdot f, \text{ i.e.\ the `fiber degree',}
\end{align*}
together give a slope-like function $\mu_f(-) := C_1(-)/C_0(-)$ on  $\Coh (X)$.

The Harder-Narasimhan filtration for coherent sheaves associated to the slope-like function $\mu_f$ was formulated earlier by Yoshioka \cite[Section 3.3]{YosPII}, and was called the `relative Harder-Narasimhan filtration'.  It is a slightly different formulation of the relative Harder-Narasimhan filtration for coherent sheaves with respect to a projective morphism found in \cite[Theorem 2.3.2]{HL10}, for example.
\end{example}

%

\begin{remark}\label{paper13-remark1}
Our notion of a slope-like function is very close to the notion of a `weak central charge' introduced in Brown-Shipman \cite{BS}, in the following sense: given a noetherian abelian category $\mathcal A$, suppose $C_0 : K(\mathcal A) \to \mathbb{Z}$ and $C_1 : K(\mathcal A) \to \mathbb{R}$ are two functions that satisfying the positivity properties above and give rise to a slope-like function $\mu (-) := C_1(-)/C_0(-)$.  Then the group homomorphism
\begin{align*}
  Z : K(\mathcal A) &\longrightarrow \mathbb{C}, \\
  E &\longmapsto C_1(E) + iC_0(E)
\end{align*}
is a weak central charge in the sense of \cite{BS}.

Note that, in the definition of a weak central charge, the function $C_1$ (the real part of $Z$) is allowed to take values in $\mathbb{R}$, whereas we require it to take values in $\mathbb{Z}$ in the definition of a slope-like function.
\end{remark}

We finish this section with a re-interpretation of the torsion class $\mathfrak T_X'$ in terms of the slope-like function $\mu_f$.

\begin{lemma}\label{lemma10}
Let $\pi : X \to S$ and $\wh{\pi} : Y \to S$ be dual elliptic fibrations.  Suppose $E \in \Coh (X)$ is  torsion-free and $\mu_f$-semistable.  Then $E|_s$ is slope semistable for a general closed point $s \in S$, and
\begin{itemize}
\item if $\mu_f (E)>0$, then $E \in W_{0,X}'$;
\item if $\mu_f (E)=0$, then $E \in \mathcal F_X \cap \Phi ( W_{0,Y} \cap \Coh^{\leq d-1})$;
\item if $\mu_f (E)<0$, then $E \in W_{1,X}$.
\end{itemize}
\end{lemma}

\begin{proof}
We claim that the restriction $E|_s$ is slope semistable for a general point $s \in S$.  For, if not, then from the relative HN filtration for $F$ with respect to slope stability and with respect to the fibration $\pi : X \to S$, we have a dense open subset $U \subseteq S$ and a subsheaf $G \subseteq E|_U$ on $\Coh (U)$ such that $\mu ( G|_s) > \mu (E|_s)$ for all $s \in U$.  Then we can extend the inclusion $G \subseteq E|_U$ in $\Coh (U)$ to an inclusion $\overline{G} \subseteq E$ in $\Coh (X)$, and now we have $\mu_f (\overline{G}) > \mu_f (E)$, contradicting the $\mu_f$-semistability of $E$.  Thus $E|_s$ is slope semistable for a general point $s \in S$.

Let $G_0 \subseteq G_1 \subseteq G_2 := E$ be the filtration of $E$ as in \eqref{eq44}.  By Proposition \ref{pro6}(b), we know that only one of the subfactors $G_0, G_1/G_0, G_2/G_1$ can be nonzero.
\end{proof}

We now have the following alternative description of $\mathfrak T_X'$, which was defined in \eqref{eq48}:

\begin{lemma}\label{lemma19}
$\mathfrak T_X' = \{ E \in \Coh (X): \mu_{f,min}(E) \geq 0\}$.
\end{lemma}

\begin{proof}
Let $\mathcal T$ denote the category of coherent sheaves $E$ on $X$ satisfying $\mu_{f,min}(E) \geq 0$.  The inclusion $\mathcal T \subseteq \mathfrak T_X'$ follows from the existence of the $\mu_f$-HN filtration in Proposition \ref{pro5} and Lemma \ref{lemma10}.

To see the other inclusion, take any $E \in \mathfrak T_X'$.  We can assume that $E$ is torsion-free.  By Proposition \ref{pro6}(b), for a general closed point $s \in S$, all the HN factors of $E|_s$ have $\mu \geq 0$.  Now, if $\mu_{f,min}(E) <0$, then we can find a nonzero surjection $\alpha : E \twoheadrightarrow E'$ in $\Coh (X)$ where $E'$ is $\mu_f$-semistable, hence torsion-free, and with $\mu_f (E')<0$.  Besides, we know  $E'|_s$ is $\mu$-semistable for a general closed point $s \in S$ from Lemma \ref{lemma10}.  However, this implies that the restriction $\alpha |_s$ vanishes for a general closed point $s \in S$, contradicting $E'$ being torsion-free.    Thus $\mu_{f,min}(E) \geq 0$, i.e.\ $E \in \mathcal T$.
\end{proof}

\begin{remark}\label{remark3}
Lemma \ref{lemma19} means that, when $X$ is an elliptic surface, $\mathfrak T_X'$ coincides with the category $\overline{\mathfrak{T}}$ defined in  \cite[Definition 3.3.1]{YosPII}.
\end{remark}


\section{A Fourier-Mukai transform on the trivial elliptic fibration}\label{sec-FMTcohom}

From this section on, we restrict our discussion to a particularly simple case -- the trivial elliptic fibration over a projective K3 surface. There are many advantages of looking at this simple case. First of all, the action of the Fourier-Mukai transform on the cohomology classes obeys a clean formula which is easy to determine. Secondly, even in this simple example, we can already see how different it is between the Fourier-Mukai transform on a single abelian variety and that on a family. We will show later a few interesting results identifying different semistabilities via Fourier-Mukai transforms on trivial families. Before doing that, we first describe how Chern classes change under the Fourier-Mukai transform.

\subsection{Fourier-Mukai transforms on trivial families}\label{stabdyn-sec-product}

We introduce our notations. Let $C$ be an elliptic curve principally polarized by $L$; that is, $L$ is a degree $1$ line bundle on $C$. Let $S$ be a K3 surface of Picard number $1$ with $\pic(S)=\bZ [H_S]$. We write the product $X = C \times S$.

A Fourier-Mukai transform on $X$ can be considered as a trivial family of Fourier-Mukai transform on $C$ parametrised by $S$. We briefly recall the theory on a single elliptic curve $C$. We have an integral functor
\begin{equation}
\label{eqn:fm_curve}
\Phi_{\cE_C}: D^b(C) \lra D^b(C),
\end{equation}
whose kernel is given by $\cE_C=\pi_1^*L^{-1} \otimes \pi_2^*L^{-1} \otimes m^*L$ where $m$ is the addition on $C$. $\Phi_{\cE_C}$ has some properties which are summarised in the following lemma.

\begin{lemma}
\label{lem:FM_curve}
The integral functor $\Phi_{\cE_C}$ satisfies the following
\begin{enumerate}
\item The functor \eqref{eqn:fm_curve} is an equivalence of categories, hence a Fourier-Mukai transform;
\item For any $E \in D^b(C)$, we have
$$(\Phi_{\cE_C} \circ \Phi_{\cE_C}) (E) = \iota^*E[-1],$$
where $\iota$ is the ``inverse" operation with respect to the group structure on $C$;
\item The cohomological Fourier-Mukai transform
$$\Phi^H_{\cE_C}: H^*_{\mathrm{alg}}(C, \bZ) \lra H^*_{\mathrm{alg}}(C, \bZ)$$
is given by the formula
$$\Phi^H_{\cE_C}(a_0, a_1)=(a_1, -a_0),$$
considered as elements in $H^0(C, \bZ) \oplus H^2(C, \bZ)$.
\end{enumerate}
\end{lemma}

\begin{proof}
The first two statements can be found in \cite[Proposition 9.19]{Huy06}. The third statement is \cite[Lemma 9.23]{Huy06}.
\end{proof}

Now we consider the relative version. Note that we have the following fiber diagram
\begin{equation}
\label{eqn:trivial_product}
\xymatrix{
C \times C \times S \ar[r]^{\pi_{23}} \ar[d]_{\pi_{13}} & C \times S \ar[d]^{\pi_S} \\
C \times S \ar[r]_{\pi_S} & S.
}
\end{equation}

Similarly we write $\pi_{12}: C \times C \times S \lra C \times C$ for the projection to the first two factors, and write $\cE=\pi_{12}^*\cE_C$. Then $\cE$ defines a relative integral functor
\begin{equation}
\label{eqn:fm-relative}
\Phi_\cE: D^b(C \times S) \lra D^b(C \times S).
\end{equation}
It is not difficult to see that $\Phi_\cE$ is an equivalence.

\begin{proposition}
\label{prop:FM_product}
The integral functor $\Phi_\cE$ satisfies the following
\begin{enumerate}
\item The functor \eqref{eqn:fm-relative} is an equivalence of categories, hence a Fourier-Mukai transform;
\item For any $E \in D^b(C \times S)$, we have
$$(\Phi_\cE \circ \Phi_\cE) (E) = \iota^*E[-1],$$
where $\iota$ is the ``inverse" operation on $C$ and identity on $S$.
\end{enumerate}
\end{proposition}

\begin{proof}
It suffices to prove the second statement. By Lemma \ref{lem:FM_curve}, we know that $\Phi_{\cE_C} \circ \Phi_{\cE_C}$ is in fact given by $\Phi_{\cO_{\Delta_C}[-1]}$ where $\Delta_C= \{ (x, -x): x \in C \} \subset C \times C$ is the antidiagonal of the product. By pulling back to the triple product $C \times C \times S$, we similarly have
\begin{equation*}
\Phi_\cE \circ \Phi_\cE = \Phi_{\cO_\Delta}[-1]
\end{equation*}
where $\Delta=\Delta_C \times S \subset C \times C \times S$ is the relative antidiagonal in the diagram \eqref{eqn:trivial_product}. This proves the statement.
\end{proof}

\subsection{Cohomological Fourier-Mukai tranforms}

Now we turn to understanding  how Chern classes change under the Fourier-Mukai transform $\Phi_\cE$. We denote the projections from $X = C \times S$ to the two factors by $\pi_C$ and $\pi_S$. Then by the K\"unneth formula, the algebraic cohomology of $X$ can be given by
\begin{equation*}
\label{eqn:kuenneth}
H^*_{\mathrm{alg}}(X,\bZ)=H^*_{\mathrm{alg}}(C,\bZ) \otimes H^*_{\mathrm{alg}}(S,\bZ).
\end{equation*}

We note that
\begin{align*}
H^*_{\mathrm{alg}}(C,\bZ) &= H^0(C, \bZ) \oplus H^2(C, \bZ) \\
H^*_{\mathrm{alg}}(S,\bZ) &= H^0(S, \bZ) \oplus \bZ[H_S] \oplus H^4(S, \bZ)
\end{align*}
are both direct sums of $1$-dimensional sublattices. Let $e_0, e_1$ and respectively $f_0, f_1, f_2$ be the generators of these $1$-dmensional lattices, then $e_i \otimes f_j$ form an integral basis of $H^*_{\mathrm{alg}}(X,\bZ)$. For any $F \in D^b(C \times S)$, its Chern character will be given by $6$ integers under this basis, say
\begin{equation}
\label{eqn:normal_chern}
\chern(F)= \sum_{i,j}a_{ij} e_i\otimes f_j.
\end{equation}
And for the convenience of visualisation, we organize them in a matrix. By abuse of notation, we write
\begin{equation}
\label{eqn:matrix}
\chern(F)=
\begin{pmatrix}
a_{00} & a_{01} & a_{02} \\
a_{10} & a_{11} & a_{12}
\end{pmatrix}.
\end{equation}
Note that $a_{ij}$ contributes the class $a_{ij}(e_i\otimes f_j)$ to the $(i+j)$-th component of $\chern(F)$. The following two observations will help us determine the image of $\chern(F)$ under the cohomological Fourier-Mukai transform $\Phi^H_\cE$.

\begin{lemma}
\label{lem:FM_box}
For any $F_C \in D^b(C)$ and $F_S \in D^b(S)$, we have
\begin{equation*}
\Phi_\cE(\pi_C^*F_C \otimes \pi_S^*F_S) =  \pi_C^*\Phi_{\cE_C}(F_C) \otimes \pi_S^*F_S.
\end{equation*}
\end{lemma}

\begin{proof}
We first show the case when $F_S=\cO_S$ is trivial. This is straightforward, since $\pi_C^*F_C$ is a trivial family of sheaves $F_C$ parametrized by $S$. The relative Fourier-Mukai transform $\Phi_\cE$ takes it to a trivial family of sheaves $\Phi_{\cE_C}(F_C)$ parametrized by the same base $S$. That is
\begin{equation*}
\Phi_\cE(\pi_C^*F_C) = \pi_C^*\Phi_{\cE_C}(F_C).
\end{equation*}

To show the general case, we first observe that
$$ \pi_{13}^*\pi_S^*F_S = \pi_3^*F_S = \pi_{23}^*\pi_S^*F_S. $$
Together with the projection formula, we have
\begin{align*}
&\ \Phi_\cE(\pi_C^*F_C \otimes \pi_S^*F_S) \\
= &\ \pi_{23*}(\cE \otimes \pi_{13}^*\pi_C^*F_C \otimes \pi_{13}^*\pi_S^*F_S) \\
= &\ \pi_{23*}(\cE \otimes \pi_{13}^*\pi_C^*F_C \otimes \pi_{23}^*\pi_S^*F_S) \\
= &\ \pi_{23*}(\cE \otimes \pi_{13}^*\pi_C^*F_C) \otimes \pi_S^*F_S \\
= &\ \Phi_\cE(\pi_C^*F_C) \otimes \pi_S^*F_S \\
= &\ \pi_C^*\Phi_{\cE_C}(F_C) \otimes \pi_S^*F_S.
\end{align*}
This finishes the proof.
\end{proof}

\begin{lemma}
\label{lem:functorial}
For any $F_C \in D^b(C)$ and $F_S \in D^b(S)$, we have
\begin{equation*}
\chern(\pi_C^*F_C \otimes \pi_S^*F_S) = \pi_C^* \chern(F_C) \otimes \pi_S^* \chern(F_S).
\end{equation*}
\end{lemma}

\begin{proof}
This is clear since taking Chern characters commutes with tensor products and pullbacks.
\end{proof}

Finally, we are ready to compute the cohomological Fourier-Mukai transform
\begin{equation*}
\Phi^H_\cE: H^*_{\mathrm{alg}}(X, \bZ) \lra H^*_{\mathrm{alg}}(X, \bZ).
\end{equation*}
Note that $\Phi^H_\cE$ is a linear map. Therefore it suffices to determine the images of classes in a basis of $H^*_{\mathrm{alg}}(X, \bZ)$, which can be achieved by making appropriate choices of $F_C$ and $F_S$.

\begin{proposition}
\label{prop:product_change_class}
For any $F \in D^b(X)$, following the notation in \eqref{eqn:matrix} we write
\begin{equation*}
\chern(F)=
\begin{pmatrix}
a_{00} & a_{01} & a_{02} \\
a_{10} & a_{11} & a_{12}
\end{pmatrix}.
\end{equation*}
Then the Chern character of $\Phi_\cE(F)$ is given by
\begin{equation*}
\chern(\Phi_\cE(F))=
\begin{pmatrix}
a_{10} & a_{11} & a_{12} \\
-a_{00} & -a_{01} & -a_{02}
\end{pmatrix}.
\end{equation*}
\end{proposition}

\begin{proof}
To apply the above lemma to deduce the formula, we choose candidates for $F_C$ and $F_S$ as follows.
\begin{itemize}
\item Let $F^0_C=\cO_C$ be the structure sheaf, then $\chern(F^0_C)=(1,0)$;
\item Let $F^1_C=\cO_y$ be a skyscraper sheaf of any closed point $y\in C$, then $\chern(F^1_S)=(0,1)$;
\item Let $F^0_S=\cO_S$ be the structure sheaf, then $\chern(F^0_S)=(1,0,0)$;
\item Pick a smooth curve in the linear system $D \in |H_S|$, and let $F^1_S$ be a line bundle on $D$ of degree $\frac{H^2}{2}$, then it is easy to find that $\chern(F^1_S)=(0,1,0)$;
\item Let $F^2_S=\cO_x$ be a skyscraper sheaf of any closed point $x\in S$, then $\chern(F^2_S)=(0,0,1)$.
\end{itemize}
By applying Lemma \ref{lem:functorial}, it is straightforward to check that the following is true:
\begin{equation*}
\chern(\pi_C^*F^i_C \otimes \pi_S^*F^j_S)=e_i\otimes f_j
\end{equation*}
under the notation in \eqref{eqn:normal_chern}. That is, one entry is ``$1$" and all the other entries are ``$0$" in the matrix form of the Chern character in \eqref{eqn:matrix}.

Now we look at the Chern characters after applying the Fourier-Mukai transform $\Phi_\cE$. First of all, by Lemma \ref{lem:FM_curve}, we have that $\chern(\Phi_{\cE_C}(F_C^0))=(0,-1)$ and $\chern(\Phi_{\cE_C}(F_C^1))=(1,0)$.

We can also combine Lemma \ref{lem:FM_box} and Lemma \ref{lem:functorial} to get
\begin{align*}
\chern(\Phi_\cE(\pi_C^*F_C \otimes \pi_S^*F_S)) &= \chern (\pi_C^*\Phi_{\cE_C}(F_C) \otimes \pi_S^*F_S) \\
&= \pi_C^*\chern(\Phi_{\cE_C}(F_C)) \otimes \pi_S^*\chern(F_S).
\end{align*}
Applying this formula, it is straightforward to check that the following is true:
\begin{equation*}
\chern(\Phi_\cE(\pi_C^*F^i_C \otimes \pi_S^*F^j_S))=(-1)^{i+1} e_{1-i} \otimes f_j
\end{equation*}
under the notation in \eqref{eqn:normal_chern}. In the matrix form, this verifies the formula in the statement for objects $\pi_C^*F^i_C \otimes \pi_S^*F^j_S \in D^b(X)$. Since their Chern characters form a basis of $H^*_{\mathrm{alg}}(X,\bZ)$, this shows that
\begin{equation*}
\Phi^H_\cE(\begin{pmatrix}
a_{00} & a_{01} & a_{02} \\
a_{10} & a_{11} & a_{12}
\end{pmatrix})=\begin{pmatrix}
a_{10} & a_{11} & a_{12} \\
-a_{00} & -a_{01} & -a_{02}
\end{pmatrix},
\end{equation*}
as desired.
\end{proof}

We remind the reader again that the degrees of various components are slightly mixed up. For instance, $a_{01}$ moves from $\chern_1$ to $\chern_0$ after the Fourier-Mukai transform, while $a_{10}$ moves from $\chern_1$ to $\chern_2$, etc. This can be easily seen from the matrix form, as entries lying on the same antidiagonal represent classes in the same degree.

\subsection{Product of cohomology classes}\label{sec-intersecpro}

For later convenience we compute the cup products of cohomology classes on $X=C\times S$. We will also give a formula for computing slopes with respect to any ample class on $X$. We follow the notations in the previous subsection. In particular, we will use alternatingly the two ways for writing the Chern classes in \eqref{eqn:normal_chern} and \eqref{eqn:matrix}.

We have seen that $e_i \otimes f_j$ for $0 \leq i \leq 1$, $0 \leq j \leq 2$ form an integral basis of $H^*_{\mathrm{alg}}(X, \bZ)$. To compute arbitrary cup products of classes in mixed degrees, it suffices to compute intersections of classes in this basis, which is given by the following lemma.

\begin{lemma}
Following the above notations, we have
\begin{equation*}
(e_{i_1} \otimes f_{j_1}) \cup (e_{i_2} \otimes f_{j_2}) = (e_{i_1} \cup e_{i_2}) \otimes (f_{j_1} \cup f_{j_2}),
\end{equation*}
where the cup products on the right hand side are the usual cup products on $C$ and $S$, respectively. In particular, the above product is always an integral multiple of $e_{i_1+i_2} \otimes f_{j_1+j_2}$.
\end{lemma}

\begin{proof}
This is immediate since the K\"unneth formula gives a ring isomorphism.
\end{proof}

We can write the product more explicitly in the following way. Note that the degree $2$ and degree $4$ components of $H^*_{\mathrm{alg}}(X, \bZ)$ are given by
\begin{align*}
\ns(X) = H^2_{\mathrm{alg}}(X, \bZ) &= \bZ (e_0 \otimes f_1) \oplus \bZ (e_1 \otimes f_0), \\
H^4_{\mathrm{alg}}(X, \bZ) &= \bZ (e_0 \otimes f_2) \oplus \bZ (e_1 \otimes f_1),
\end{align*}
respectively. For simplicity, we further assume that $H_S^2=2d$ (as an intersection of divisor classes on the K3 surface $S$).

The intersections of two divisor classes are given by
\begin{align*}
(e_0 \otimes f_1) \cup (e_0 \otimes f_1) &= 2d \cdot (e_0 \otimes f_2), \\
(e_0 \otimes f_1) \cup (e_1 \otimes f_0) &= e_1 \otimes f_1, \\
(e_1 \otimes f_0) \cup (e_1 \otimes f_0) &= 0.
\end{align*}
And the intersections of a divisor class and a curve class are given by
\begin{align*}
(e_0 \otimes f_1) \cup (e_0 \otimes f_2) &= 0, \\
(e_0 \otimes f_1) \cup (e_1 \otimes f_1) &= 2d, \\
(e_1 \otimes f_0) \cup (e_0 \otimes f_2) &= 1, \\
(e_1 \otimes f_0) \cup (e_1 \otimes f_1) &= 0.
\end{align*}
And the only non-trivial triple intersection of divisor classes under the above basis is given by
\begin{equation}
\label{eqn:triple}
(e_0 \otimes f_1) \cup (e_0 \otimes f_1) \cup (e_1 \otimes f_0) = 2d.
\end{equation}

Now we determine the ample cone of $X$.

\begin{lemma}
\label{lem:ample-cone}
The ample cone of $X$ is given by
\begin{equation*}
\amp(X) = \{ \alpha (e_0 \otimes f_1) + \beta (e_1 \otimes f_0) \mid \alpha > 0, \beta > 0 \}.
\end{equation*}
\end{lemma}

\begin{proof}
This is an easy consequence of the Nakai-Moischezon criterion.
\end{proof}

Now we write down a formula for computing the slope with respect to an arbitrary ample class.

\begin{lemma}
\label{lem:slope}
Let $H=\alpha (e_0 \otimes f_1) + \beta (e_1 \otimes f_0)$ for some $\alpha>0$ and $\beta>0$  be an ample class on $X$. Let $F$ be a sheaf on $X$ whose Chern class is given by \eqref{eqn:normal_chern}. Then its slope with respect to $H$ is given by
\begin{equation*}
\mu_H(F) = 2d \cdot \left( \frac{a_{10}}{a_{00}} \cdot \alpha^2 + \frac{a_{01}}{a_{00}} \cdot 2\alpha\beta \right).
\end{equation*}
\end{lemma}

\begin{proof}
By definition we have
\begin{align*}
\mu_H(F) &= \frac{c_1(F) \cdot H^2}{\rank(F)} \\
&= \frac{1}{a_{00}} \cdot (a_{01}(e_0 \otimes f_1) + a_{10}(e_1 \otimes f_0)) \cdot (\alpha (e_0 \otimes f_1) + \beta (e_1 \otimes f_0))^2.
\end{align*}
From here one can easily apply \eqref{eqn:triple} to get the desired formula.
\end{proof}

We immediately have the following formula for the slope of the image of $F$ under the Fourier-Mukai transform constructed in Proposition \ref{prop:FM_product}.

\begin{corollary}
\label{cor:slope}
Assume that $\Phi(F)$ is a sheaf on $X$, then its slope with respect to an ample line bundle $H'=\alpha' (e_0 \otimes f_1) + \beta' (e_1 \otimes f_0)$ for some $\alpha'>0$ and $\beta'>0$ is given by
\begin{equation*}
\mu_{H'}(\Phi(F)) = 2d \cdot \left( \frac{-a_{00}}{a_{10}} \cdot \alpha'^2 + \frac{a_{11}}{a_{10}} \cdot 2\alpha'\beta' \right).
\end{equation*}
\end{corollary}

\begin{proof}
This is a combination of Proposition \ref{prop:product_change_class} and Lemma \ref{lem:slope}.
\end{proof}

From Lemma \ref{lem:slope} and Corollary \ref{cor:slope}, we see that if $\alpha \gg \beta$, then the polarization is ``fiber-like". That is, the slope it produces will be very close to the slope of the restriction of $F$ on a generic fiber. However, due to the mixture of the contributions from the base and the fiber, the slopes of $F$ and $\Phi(F)$ do not follow any simple relation, which makes the Fourier-Mukai transforms on families more difficult to handle.

\section{Preservation of semistability on elliptic threefolds}
\label{sec:pres}




\subsection{Hilbert polynomials of coherent sheaves}

In this section, let $X = C \times S$ where $C$ is an elliptic curve and $S$ is a projective K3 surface of Picard rank $1$.  For a coherent sheaf $E$ on $X$, its Hilbert polynomial with respect to the polarisation $\omega := H+nD$ (where $n$ is some positive integer) is given by
\begin{align*}
  \chi (E \otimes & \cO_X (m(H+nD))) \\
  &= \int_X \ch(E) \cdot \ch (\cO_X (m(H+nD)))\cdot \td (X) \\
&= \frac{m^3}{6} \cdot \rk (E) + \frac{m^2}{2} (H+nD)^2 (\ch_1(E)+\rk(E)\cdot \td_1(X)) \\
 &\text{\quad} m(H+nD) (\ch_2(E)+\ch_1(E)\cdot \td_1(X) + \rk(E) + \td_2(X)) + \chi (E).
\end{align*}
Therefore, when $\dimension E = 3$, the  reduced Hilbert polynomial of $E$ is given by
\begin{align*}
  \frac{m^3}{6} (H+nD)^3 &+ \frac{m^2}{2}(H+nD)^2 \left( \frac{\ch_1(E)}{\rk(E)} + \td_1(X)\right) \\
  & + m(H+nD)\left( \frac{\ch_2(E)}{\rk(E)} + \frac{\ch_1(E)\cdot \td_1(X)}{\rk(E)} + \td_2(X)\right) \frac{\chi(E)}{\rk(E)}.
\end{align*}
Ignoring the terms independent of $E$ and noting that $H^2=0$, we see that $\omega$-semistability of $E$, for $n \gg 0$,  is determined by lexicographical ordering with respect to the following vector and the corresponding vectors for subsheaves of $E$:

\begin{multline}
\left( \frac{\ch_1(E)}{\rk(E)} \cdot D^2, \quad \frac{\ch_1(E)}{\rk(E)} \cdot H\cdot D, \quad \left( \frac{\ch_1(E)}{\rk(E)} \cdot \td_1(X) + \frac{\ch_2(E)}{\rk(E)} \right) \cdot D, \qquad\qquad\qquad\qquad\right. \\
\left. \left( \frac{\ch_1(E)}{\rk(E)} \cdot \td_1(X) + \frac{\ch_2(E)}{\rk(E)} \right) \cdot H, \quad \frac{\chi(E)}{\rk(E)} \right).\label{0602-eq6}
\end{multline}

When $\dimension E=2$, the Hilbert polynomial of $E$ is
\[
  \frac{m^2}{2} (H+nD)^2 \ch_1(E) + m(H+nD)(\ch_2(E)+\ch_1(E)\cdot \td_1(X)) + \chi (E),
\]
and so the reduced Hibert polynomial is
\begin{equation}\label{0602-eq7}
 \frac{m^2}{2} + m \frac{ (H+nD)(\ch_2(E)+\ch_1(E)\cdot \td_1(X))}{(H+nD)^2 \ch_1(E)} + \frac{\chi (E)}{(H+nD)^2\ch_1(E)}.
\end{equation}
Ignoring the terms independent of $E$ and using $H^2=0$, and since $\td_1(X)=0$, we are left with
\[
 m \frac{ (H+nD)\ch_2(E)}{ (n\cdot 2HD + n^2D^2)\ch_1(E)} + \frac{\chi(E)}{ (n\cdot 2HD + n^2D^2)\ch_1(E)}.
\]
Therefore, if $E$  is $\omega$-semistable for $n \gg 0$, then   for any subsheaf $0 \neq E' \subseteq E$ we must have
\begin{equation}\label{0602-eq1}
  \frac{ (H+nD)\ch_2(E')}{ (n\cdot 2HD + n^2D^2)\ch_1(E')} \leq \frac{ (H+nD)\ch_2(E)}{ (n\cdot 2HD + n^2D^2)\ch_1(E)} \text{\quad for $n \gg 0$}.
\end{equation}
 In the case that the denominators on both sides of \eqref{0602-eq1} are nonzero, we have:
 \begin{itemize}
 \item The inequality \eqref{0602-eq1} is equivalent to
\begin{multline}
n^3 \cdot D \ch_2 (E') \cdot D^2 \ch_1(E) + n^2 \cdot \left( D \ch_2(E') \cdot 2HD  \ch_1(E) + H \ch_2(E')\cdot D^2 \ch_1(E)\right) \\
+ n \left( H  \ch_2(E') \cdot 2HD \ch_1(E)\right) \\
\leq n^3 \cdot D \ch_2 (E) \cdot D^2 \ch_1(E') + n^2 \cdot \left( D  \ch_2(E) \cdot 2HD  \ch_1(E') + H \ch_2(E)\cdot D^2 \ch_1(E')\right) \\
+ n \left( H  \ch_2(E) \cdot 2HD \ch_1(E')\right).\label{0602-eq2}
\end{multline}
Then,  a sufficient condition for  \eqref{0602-eq2} to hold for $n \gg 0$  is to have strict inequality for the leading coefficients on both sides, i.e.\
\begin{equation}\label{0602-eq3}
\frac{\ch_2(E')\cdot D}{\ch_1(E')\cdot D^2} < \frac{\ch_2(E)\cdot D}{\ch_1(E)\cdot D^2}.
\end{equation}
\item In the case of equality in \eqref{0602-eq3}, a sufficient condition for \eqref{0602-eq2} to hold for $n \gg 0$ is to have strict inequality for the coefficients of the $n^2$ terms, i.e.\
\begin{multline}\label{0602-eq4}
\left( D  \ch_2(E') \cdot 2HD  \ch_1(E) + H  \ch_2(E')\cdot D^2 \ch_1(E)\right) \\
< \left( D  \ch_2(E) \cdot 2HD  \ch_1(E') + H  \ch_2(E)\cdot D^2 \ch_1(E')\right),
\end{multline}
which is equivalent to
\begin{multline}
D^2\ch_1 (E') \left( \frac{D\ch_2(E')}{D^2 \ch_1(E')} \cdot 2HD \ch_1(E)-H\ch_2(E)\right) \notag\\
< D^2 \ch_1 (E) \left( \frac{D\ch_2(E)}{D^2\ch_1(E)} 2HD \ch_1 (E') - H\ch_2(E')\right). \label{0602-eq5}
\end{multline}
\end{itemize}

%

\subsection{Positivity of components of Chern character}

Now we turn to proving various positivity results. To motivate this, note that any torsion-free coherent sheaf $F$ on a smooth projective variety $X$ must have positive rank. In general, the component of $\ch(F)$ of the lowest degree always has similar effectivity property. For later convenience, here we collect some results along this line for any coherent sheaf $F$ on $X=C \times S$.


In the following, we always use the notation given by \eqref{eqn:matrix}. In other words, for any $F \in \coh(X)$, the term containing $e_i \otimes f_j$ in $\ch(F)$ is denoted by
$$ \ch_{ij}(F) = a_{ij} e_i \otimes f_j. $$
Then we write $\ch(F)$ in the form of a matrix
\begin{equation*}
\ch(F) = \begin{pmatrix} a_{00} & a_{01} & a_{02} \\
a_{10} & a_{11} & a_{12} \end{pmatrix}.
\end{equation*}

We first of all establish two formulas concerning Chern classes of restrictions of a sheaf $F$ to a generic section or a generic fiber. They will be useful in the proofs of various positivity results.

\begin{lemma}\label{lem:pos0}
Let $X = C \times S$. For any $F \in \coh(X)$, assume $\ch(F)$ is given by \eqref{eqn:matrix}. For a generic closed point $c \in C$, the Chern character of the restriction $F|_{\{c\} \times S}$ is given by
\begin{equation}
\label{eqn:res-to-section}
\ch(F|_{\{c\} \times S}) = a_{00}f_0 + a_{01}f_1 + a_{02}f_2 \in H^*_{\mathrm{alg}}(S, \bZ).
\end{equation}
Similarly, for a generic closed point $s \in S$, the Chern class of the restriction $F|_{C \times \{s\}}$ is given by
\begin{equation}
\label{eqn:res-to-fiber}
\ch(F|_{C \times \{s\}}) = a_{00}e_0 + a_{10}e_1 \in H^*_{\mathrm{alg}}(C, \bZ).
\end{equation}
\end{lemma}

\begin{proof}
We prove \eqref{eqn:res-to-section} first. Since $F \in \coh(X)$, we can also think of $F$ as an $\cO_C$-module. By generic flatness \cite[Theorem 6.9.1]{EGA4}, there exists a non-empty open subset $U \subset C$ such that $F|_{U \times S}$ is flat over $\cO_U$. For any closed point $c \in U$, we have the exact sequence
$$ 0 \lra F \otimes_{\cO_C} \cO_C(-c) \lra F \lra F \otimes_{\cO_C} \cO_c \lra 0. $$
Or we can write it as
$$ 0 \lra F \otimes_{\cO_X} \cO_X(-\{c\} \times S) \lra F \lra F \otimes_{\cO_X}\cO_{\{c\} \times S}  \lra 0 $$
where $i: \{c\} \times S \hookrightarrow C \times S$ is the embedding. Now we compute $\ch(F \otimes_{\cO_X}\cO_{\{c\} \times S})$. By assumption we have that
$$ \ch (F)= \begin{pmatrix} a_{00} & a_{01} & a_{02} \\ a_{10} & a_{11} & a_{12} \end{pmatrix}, $$
and it is clear that
$$ \ch(\cO_{\{c\} \times S}) = \begin{pmatrix} 0 & 0 & 0 \\ 1 & 0 & 0 \end{pmatrix}. $$
Therefore we have
$$ \ch(F \otimes_{\cO_X}\cO_{\{c\} \times S}) = \begin{pmatrix} 0 & 0 & 0 \\ a_{00} & a_{01} & a_{02} \end{pmatrix}. $$
Since $F \otimes_{\cO_X}\cO_{\{c\} \times S} = i_\ast(F|_{\{c\} \times S})$, by Grothendieck-Riemann-Roch theorem, we have
$$ i_\ast(\ch(F|_{\{c\} \times S}) \cdot \td(S)) = \ch(F \otimes_{\cO_X}\cO_{\{c\} \times S}) \cdot \td(C \times S). $$
As $\td(S)$, hence $\pi^*_S\td(S)$, is invertible, the equation simplifies to
$$ i_\ast(\ch(F|_{\{c\} \times S})) = \ch(F \otimes_{\cO_X}\cO_{\{c\} \times S}) \cdot \pi^*_C\td(C) = \begin{pmatrix} a_{00} & a_{01} & a_{02} \\ a_{10} & a_{11} & a_{12} \end{pmatrix}. $$
It follows that we have
$$ \ch(F|_{\{c\} \times S}) = a_{00}f_0 + a_{01}f_1 + a_{02}f_2 \in H^*_{\mathrm{alg}}(S, \bZ). $$

The equation \eqref{eqn:res-to-fiber} can be proved in a similar way.
\end{proof}

We start to prove the positivity results. To avoid repetitions, in the rest of the section, we always assume that $X=C \times S$ is the product of an elliptic curve and a K3 surface of Picard number $1$, and $F \in \coh(X)$.

\begin{lemma}
\label{lem:pos1}
If $\ch (F)=\begin{pmatrix} 0 & a_{01} & \ast \\ a_{10} & \ast & \ast \end{pmatrix}$, then $a_{01}, a_{10} \geq 0$.
\end{lemma}

\begin{proof}
For any ample line bundle $L$ on $X$, we consider the Hilbert polynomial
$$P(m)=\chi(F \otimes L^m).$$
Since $F$ is a torsion sheaf, the polynomial $P(m)$ has degree at most $2$. By Riemann-Roch, the coefficient of the $m^2$-term is $\frac{1}{2}L^2 \cdot c_1(F)$. By Serre vanishing, when $m \gg 0$, we must have $P(m)>0$ because all higher cohomology groups of $F \otimes L^m$ vanish. Therefore we must have
$$L^2 \cdot c_1(F) \geq 0.$$

By Lemma \ref{lem:ample-cone}, we can write
$$L=\alpha e_0 \otimes f_1 + \beta e_1 \otimes f_0$$
for some $\alpha, \beta>0$. With the given assumption on $c_1(F)$, we have
$$L^2 \cdot c_1(T) = 2d (\alpha^2 a_{10} + 2\alpha\beta a_{01})$$
where $2d=(e_0 \otimes f_1) \cdot (e_0 \otimes f_1) \cdot (e_1 \otimes f_0) >0$. Therefore we must have
$$ \alpha^2 a_{10} + 2\alpha\beta a_{01} \geq 0 $$
for all $\alpha,\beta>0$. If we fix $\beta>0$ and take $\alpha \gg 0$, we get $a_{10} \geq 0$. If we fix $\alpha>0$ and take $\beta \gg 0$, we get $a_{01} \geq 0$, as required.
\end{proof}

\begin{lemma}
\label{lem:pos3}
If $\ch (F)= \begin{pmatrix} 0 & 0 & a_{02} \\ a_{10} & \ast & \ast \end{pmatrix}$, then $a_{02}, a_{10} \geq 0$.
\end{lemma}

\begin{proof}
The inequality $a_{10} \geq 0$ follows from Lemma \ref{lem:pos1}.

By formula \eqref{eqn:res-to-section}, for a generic closed point $c \in C$, we have that
$$ \ch(F_{\{c\} \times S}) = a_{00}f_0 + a_{01}f_1 + a_{02}f_2 = a_{02}f_2. $$
This implies that $F_{\{c\} \times S}$ is supported on finitely many points, and $a_{02} = \chi(F_{\{c\} \times S}) \geq 0$.
\end{proof}



\begin{lemma}
\label{lem:chern}
Let $X = C \times S$. Let $F$ be a sheaf on $X$ with Chern character given by
\begin{equation}
\label{eqn:chern}
\ch(F) = \begin{pmatrix}
0 & 0 & \ast \\
\ast & \ast & \ast
\end{pmatrix},
\end{equation}
then $F|_{C \times \{s\}}$ has dimension $0$ for all closed points $s \in S$ except finitely many of them.
\end{lemma}

We will need the following two results in the proof.

\begin{sublemma}
\label{sublem:chern1}
Let $F$ be a sheaf on $X$ with Chern character given by \eqref{eqn:chern}, then for every curve $D \subset S$, $F|_{C \times D}$ is a torsion sheaf.
\end{sublemma}

\begin{proof}
We write the inclusion
$$ i: C \times D \hookrightarrow C \times S, $$
then $F|_{C \times D} = i^*F$. We need to show that $\rk(i^*F)=0$.

Since $i^*F$ is the restriction of $F$ on $C \times D$, we have an exact sequence
$$ 0 \lra K \lra F \lra i_*i^*F \lra 0. $$
By condition \eqref{eqn:chern} we know that $F$ is a torsion sheaf, hence so is $K$, therefore $\ch_1(K)$ must be effective.

We first compute $\ch(i_*i^*F)$. We apply Grothendieck-Riemann-Roch on the closed immersion $i$
$$ \ch(i_*i^*F)=i_*(\ch(i^*F) \cdot \td(N_{C \times D/C \times S}^\vee)), $$
where $N_{C \times D/C \times S}^\vee$ is the conormal bundle of $C \times D$ in $C \times S$. Notice that both $i^*F$ and $N_{C \times D/C \times S}^\vee$ are sheaves on $C \times D$, hence we have
\begin{align*}
\ch(i^*F) &= (\rk(i^*F), \ast, \ast); \\
\td(N_{C \times D/C \times S}^\vee) &= (1, \ast, \ast).
\end{align*}
Therefore we have
$$ \ch(i^*F) \cdot \td(N_{C \times D/C \times S}^\vee) = (\rk(i^*F), a_1, a_2), $$
where $a_1 \in \HH^2(C \times D)$ and $a_2 \in \HH^4(C \times D)$. After pushforward, we get
$$ i_*(\ch(i^*F) \cdot \td(N_{C \times D/C \times S}^\vee)) = i_*(\rk(i^*F), a_1, a_2). $$
Notice that we have
$$ i_*(\rk(i^*F)) = \rk(i^*F)\mathrm{PD}(C \times D) = \rk(i^*F) \deg(D) \cdot e_0 \otimes f_1, $$
where $\mathrm{PD}$ is the Poincar\'{e} dual, and $\deg(D)$ is the degree of the curve $D$ in $S$.
Moreover, we have
\begin{align*}
i_*(a_1) &\in \HH^4(C \times S); \\
i_*(a_1) &\in \HH^6(C \times S).
\end{align*}
Therefore in the matrix notation, we have
$$ \ch(i_*i^*F) = i_*(\rk(i^*F), a_1, a_2) = \begin{pmatrix}
0 & \rk(i^*F) \deg(D) & \ast \\
0 & \ast & \ast
\end{pmatrix}. $$
It follows that
$$ \ch(K) = \ch(F)-\ch(i_*i^*F) = \begin{pmatrix}
0 & -\rk(i^*F) \deg(D) & \ast \\
\ast & \ast & \ast
\end{pmatrix}. $$
If $\rk(i^*F)>0$, then since $\deg(D)>0$, we have $-\rk(i^*F) \deg(D)<0$, hence $\ch_1(K)$ is not effective, which contradicts the fact that $K$ is a torsion sheaf. Therefore we conclude that $\rk(i^*F)>0$, which means that $F|_{C \times D}$ is a torsion sheaf.
\end{proof}

\begin{sublemma}
\label{sublem:chern2}
Let $F$ be a sheaf on $X$ such that $F|_{C \times D}$ is a torsion sheaf for every curve $D \subset S$. Then $F|_{C \times \{s\}}$ has dimension $0$ for all closed points $s \in S$ except finitely many of them.
\end{sublemma}

\begin{proof}
Let $\eta$ be the generic point of $C$, and write $T = \pi_S(\supp(F|_{\eta \times S}))$, where $\pi_S$ is the projection to the second factor. Then for any closed point $s \in S$, $F|_{C \times \{s\}}$ has dimension $0$ if and only if $s \notin T$. Therefore it suffices to show that $T$ contains only finitely many closed points in $S$.

If not, then $\dim T \geq 1$, hence $T$ contains a curve $D \subset T$. Now we consider the restriction $F|_{C \times D}$. By the definition of $T$, we know that $\eta \times T \subset \supp(F)$. Therefore its closure $C \times T \subset \supp(F)$. It follows that $C \times D \subset \supp(F)$ hence $F|_{C \times D}$ is not a torsion sheaf, which is a contradiction.
\end{proof}

\begin{proof}[Proof of Lemma \ref{lem:chern}]
The statement follows easily from Sublemmas \ref{sublem:chern1} and \ref{sublem:chern2}.
\end{proof}

\begin{lemma}
\label{lem:supp1}
If $\ch (F) = \begin{pmatrix} 0 & 0 & 0 \\ \ast & \ast & \ast \end{pmatrix}$, then $F|_{C \times \{s\}}$ is 0-dimensional for all closed points $s \in S$.
\end{lemma}

\begin{proof}
By \eqref{eqn:res-to-section}, for a generic closed point $c \in C$, the restriction $F|_{\{c\} \times S}$ has the Chern character $\ch(F|_{\{c\} \times S}) = 0$. It follows that $F$ is supported on a finite number of horizontal sections of the form $\{c\} \times S$. Therefore for any closed point $s \in S$, the restriction $F|_{C \times \{s\}}$ is supported on at most finitely many points, which are the intersections of the above sections with the fiber.
\end{proof}

\begin{lemma}
\label{lem:supp2}
If $\ch (F) = \begin{pmatrix} 0 & \ast & \ast \\ 0 & \ast & \ast \end{pmatrix}$, then $F|_{C \times \{s\}} = 0$ for a generic closed point $s \in S$.
\end{lemma}

\begin{proof}
By \eqref{eqn:res-to-fiber}, for a generic closed point $s \in S$, the restriction $F|_{C \times \{s\}}$ has the Chern character $\ch(F|_{C \times \{s\}})=0$. It follows that $F|_{C \times \{s\}}=0$ for a generic closed point $s \in S$.
\end{proof}

\begin{lemma}
\label{lem:pos2}
If $\ch (F) = \begin{pmatrix} 0 & 0 & a_{02} \\ 0 & a_{11} & \ast \end{pmatrix}$, then $a_{02}, a_{11} \geq 0$.
\end{lemma}

\begin{proof}
By Lemma \ref{lem:ample-cone}, we can write any ample line bundle $L$ on $X$ as
$$ L = \alpha e_0 \otimes f_1 + \beta e_1 \otimes f_0 $$
for some $\alpha, \beta>0$. By Serre vanishing, we have that $\chi(F \otimes L^{\otimes m}) \geq 0$ for $m \gg 0$, therefore the leading coefficient in the expansion of $\chi(F \otimes L^{\otimes m})$ must be non-negative. We find that
\begin{align*}
\chi(F \otimes L^{\otimes m}) &= \int \ch(F) \cdot \ch(L^{\otimes m}) \cdot \td(X) \\
&= \int (0, 0, \ch_{02}+\ch_{11}, \ch_{12}) \cdot \ch(L^{\otimes m}) \cdot \td(X) \\
&= m(\alpha e_0 \otimes f_1 + \beta e_1 \otimes f_0) \cdot (\ch_{02}+\ch_{11}) + \text{ constant term}.
\end{align*}
Therefore we require
$$ (\alpha e_0 \otimes f_1 + \beta e_1 \otimes f_0) \cdot (\ch_{02}+\ch_{11}) \geq 0 $$
or equivalently,
$$ \alpha \cdot a_{11} + \beta \cdot a_{02} \geq 0 $$
for any $\alpha, \beta>0$. By taking $\alpha \gg 0$ and $0 < \beta \ll 1$, the term $\alpha \cdot a_{11}$ dominates, hence we get $a_{11} \geq 0$. Similarly, by taking $0 < \alpha \ll 1$ and $\beta \gg 0$, we get $a_{02} \geq 0$.
\end{proof}

\begin{lemma}
\label{lem:supp3}
If $\ch (F) = \begin{pmatrix} 0 & 0 & \ast \\ 0 & \ast & \ast \end{pmatrix}$, then $F$ is supported in dimension at most $1$.
\end{lemma}

\begin{proof}
As $\ch_i(F)=0$ for $i=0$ and $1$, the support of $F$ has at least codimension $2$ in $X$. The claim follows.
\end{proof}

\begin{lemma}
\label{lem:supp4}
If $\ch (F) = \begin{pmatrix} 0 & 0 & a_{02} \\ 0 & 0 & a_{12} \end{pmatrix}$, then $F$ is supported on a finite number of fibers.
\end{lemma}

\begin{proof}
By Lemma \ref{lem:supp3} we know that $F$ is supported in dimension at most $1$. It remains to show that every $1$-dimensional irreducible component in the support of $F$ must be a fiber.

Assume on the contrary that there is some $1$-dimensional irreducible component in the support of $F$ whose projection to $S$ is still $1$-dimensional. Then the support of $(\pi_S)_\ast F$ is $1$-dimensional. Therefore we can write $\ch((\pi_S)_\ast F) = b_1f_1 + b_2f_2$ where $b_1 \neq 0$.

Moreover, $R^1(\pi_S)_\ast F$ is only supported at points $s \in S$ such that the entire fiber $C \times \{s\}$ is in the support of $F$. Therefore the support of $R^1(\pi_S)_\ast F$ contains only a finite number of points. It follows that $\ch(R^1(\pi_S)_\ast F) = b'_2f_2$.

On the other hand, we realize that $\td(C)=1$ hence $\pi_C^*\td(C)=1$. By Grothendieck-Riemann-Roch, we have that
\begin{align*}
\ch((\pi_S)_\ast F) - \ch(R^1(\pi_S)_\ast F) &= (\pi_S)_\ast(\ch(F) \cdot \pi_C^*\td(C)) \\
&= (\pi_S)_\ast \begin{pmatrix} 0 & 0 & a_{02} \\ 0 & 0 & a_{12} \end{pmatrix} = a_{12}f_2.
\end{align*}

Comparing the above computations we get
$$ b_1f_1 + (b_2-b'_2)f_2 = a_{12}f_2, $$
hence $b_1=0$, which is a contradiction.
\end{proof}

\begin{lemma}
\label{lem:pos4}
If $\ch (F)= \begin{pmatrix} 0 & 0 & 0 \\ 0 & 0 & a_{12} \end{pmatrix}$, then $a_{12} \geq 0$.
\end{lemma}

\begin{proof}
By assumption, we have $\ch_i(F)=0$ for $i=0,1$ and $2$. Therefore the support of $F$ is $0$-dimensional hence a finite number of points. It follows that $a_{12} = \chi(F) \geq 0$ as it is the length of $F$.
\end{proof}

To summarise the above discussion and give a complete list of positivity results that we can achieve for the Chern character of any coherent sheaf $F$ on $X=C \times S$, we state the following proposition. For simplicity, for any closed point $s \in S$, we write $F|_s$ to denote $F|_{C \times \{s\}}$.

\begin{proposition}
\label{prop:summary}
Let $X = C \times S$ where $C$ is an elliptic curve and $S$ is a K3 surface of Picard number $1$. For any $F \in \coh(X)$, we write the Chern character $\ch(F)$ in the form of \eqref{eqn:matrix}. Then we have
\begin{enumerate}
\item[(1)] If $\ch (F)=\begin{pmatrix} 0 & \ast & \ast \\ \ast & \ast & \ast \end{pmatrix}$, then $a_{01}, a_{10} \geq 0$.
\item[(2)] If $\ch_0 (F)= \begin{pmatrix} 0 & 0 & \ast \\ \ast & \ast & \ast \end{pmatrix}$, then $a_{02}, a_{10} \geq 0$, and $F|_s$ is 0-dimensional for all but a finite number of closed points $s \in S$.
\item[(3)] If $\ch (F) = \begin{pmatrix} 0 & 0 & 0 \\ \ast & \ast & \ast \end{pmatrix}$, then $a_{10} \geq 0$, and $F|_s$ is 0-dimensional for all closed points $s \in S$.
\item[(4)] If $\ch (F) = \begin{pmatrix} 0 & \ast & \ast \\ 0 & \ast & \ast \end{pmatrix}$, then $a_{01} \geq 0$, and $F|_s = 0$ for a general closed point $s \in S$.
\item[(5)] If $\ch (F) = \begin{pmatrix} 0 & 0 & \ast \\ 0 & \ast & \ast \end{pmatrix}$, then $a_{02}, a_{11} \geq 0$, and $F$ is supported in dimension at most 1.
\item[(6)] If $\ch (F) = \begin{pmatrix} 0 & 0 & \ast \\ 0 & 0 & \ast \end{pmatrix}$, then $a_{02} \geq 0$, and $F$ is supported on a finite number of fibers.
\item[(7)] If $\ch (F) = \begin{pmatrix} 0 & 0 & 0 \\ 0 & \ast & \ast \end{pmatrix}$, then $F$ is supported in dimension at most 1, and $F|_s$ is 0-dimensional for all closed points $s \in S$.
\item[(8)] If $\ch (F) = \begin{pmatrix} 0 & 0 & 0 \\ 0 & 0 & \ast \end{pmatrix}$, then $F$ is supported at a finite number of points, and $\ch_3(F) \geq 0$.
\end{enumerate}
\end{proposition}

\begin{proof}
This is a collection of statements proved in the above series of lemmas. (1) follows from Lemma \ref{lem:pos1}. (2) follows from Lemmas \ref{lem:pos3} and \ref{lem:chern}. (3) follows from Lemmas \ref{lem:pos1} and \ref{lem:supp1}. (4) follows from Lemmas \ref{lem:pos1} and \ref{lem:supp2}. (5) follows from Lemmas \ref{lem:pos2} and \ref{lem:supp3}. (6) follows from Lemmas \ref{lem:pos2} and \ref{lem:supp4}. (7) follows from Lemmas \ref{lem:supp1} and \ref{lem:supp3}. (8) follows from Lemma \ref{lem:pos4}.
\end{proof}

\begin{remark}\label{0602-remark3}
Note that, by Proposition \ref{prop:summary}(3), any coherent sheaf $F$ on $X$ with $\ch(F) = \begin{pmatrix} 0 & 0 & 0 \\ \ast & \ast & \ast \end{pmatrix}$ must lie in $\{\coh^{\leq 0}\}^\uparrow$. However, the converse is not true; the following is a counterexample. Indeed, from the proof of Lemma \ref{lem:supp1}, we know that a coherent sheaf with its Chern class given by the above formula must be supported on the union of finitely many horizontal sections.
\end{remark}

\begin{example}
Suppose $S$ is a K3 surface containing an integral nodal curve $C' \in |H|$ of geometric genus $1$. Let $C$ be the normalization of $C'$. Then $X := C \times S$ contains the graph $\Gamma$ of the composition $C \twoheadrightarrow C' \hookrightarrow S$ as a closed subvariety. Its structure sheaf $\cO_{\Gamma}$ has the Chern character of the form
\[
 \ch (\cO_\Gamma) = \begin{pmatrix} 0 & 0 & 1 \\ 0 & \ast & \ast \end{pmatrix}
\]
since by \eqref{eqn:res-to-section} the first row of the matrix represents the Chern character of $F|_{\{c\} \times S}$ for a generic closed point $c \in C$, which is a skyscraper sheaf.
\end{example}

The following observation will be useful later on:
\begin{lemma}\label{0602-lemma4}
The categories
\[
  \coh^{\mathrm{sec}}(X) := \{ F \in \coh (X) : \ch(F) = \begin{pmatrix} 0 & 0 & 0 \\ \ast & \ast & \ast \end{pmatrix} \}
\]
and
\[
  \coh^{\lrcorner}(X) := \{ F \in \coh (X) : \ch(F) = \begin{pmatrix} 0 & 0 & \ast \\ \ast & \ast & \ast \end{pmatrix} \}
\]
are Serre subcategories of $\coh(X)$. In particular, they are both torsion classes in $\coh(X)$.
\end{lemma}

\begin{proof}
Take any coherent sheaf $F$ with $\ch(F) = \begin{pmatrix} 0 & 0 & 0 \\ \ast & \ast & \ast \end{pmatrix}$, and consider  any short exact sequence $0 \to F'' \to F \to F' \to 0$ in $\coh (X)$.  We need to show that $F'', F'$ both have Chern characters of the same form.

Since $F$ is a torsion sheaf, both $F'', F'$ are also torsion.  Then, by Proposition \ref{prop:summary} (1), we have $c_1(F'')\cdot HD \geq 0$ and $c_1(F') \cdot HD \geq 0$.  However, that $0 = c_1(F)\cdot HD = c_1(F'')\cdot HD + c_1(F')\cdot HD$ implies that $c_1(F'')\cdot HD =0=c_1(F')\cdot HD$.    Similarly, by Proposition \ref{prop:summary} (2), we have that $\ch_2(F'')\cdot H = 0 = \ch_2(F')\cdot H$, thus proving that the first category is a Serre subcategory of $\coh (X)$.  The same argument shows that the second category is also a Serre subcategory.  By Lemma \ref{lem:torsion-pair}, these are torsion classes in $\Coh (X)$.
\end{proof}



\begin{remark}\label{remark1}
Note that, whenever we have a coherent sheaf $F$ in $\coh (\pi)_{\leq d}$ (in which case $F$ has codimension at least $c := 3-(d+1)$) that is $\Phi$-WIT$_1$, we have
\[
  \ch_{1,c}(F) \cdot D^{3-(c+1)} \leq 0.
\]
The reason is simply that $\ch_{1,c}(F)\cdot D^{3-(c+1)} = \ch_{0,c}(\wh{F}[-1])=-\ch_{0,c}(\wh{F})$, where $\ch_{0,c}(\wh{F})$ is nonnegative since $\wh{F}$ is also of codimension at least $c$.  For instance, when $F$ is a $\Phi$-WIT$_1$ sheaf that lies in $\coh (\pi)_1$, we have $\ch (F) = \begin{pmatrix} 0 & \ast & \ast \\ 0 & a & \ast \end{pmatrix}$ for some $a \leq 0$.  And when $F$ is a $\Phi$-WIT$_1$ fiber sheaf that is 1-dimensional, we have $\ch (F) = \begin{pmatrix} 0 & 0 & \ast \\ 0 & 0 & a \end{pmatrix}$ for some $a \leq 0$.  Similarly, for $F \in \Coh (\pi)_{\leq d}$ that is $\Phi$-WIT$_0$, we have
\[
  \ch_{1,c}(F) \cdot D^{3-(c+1)} \geq 0.
\]
\end{remark}

Another application of Proposition \ref{prop:summary} is the existence of slope-like functions other than $\mu_f$ (which was defined in Example \ref{paper13-eg2}):
\begin{itemize}
\item On the abelian category $\mathcal A := \Coh (X)$, we can define the functions
\begin{align*}
  C_0 (-) &:= \rk (-), \\
  C_1 (-) &:= \ch_{01}(-)\cdot H\cdot D.
\end{align*}
For any $F \in \Coh (X)$, that $C_1 (F) \geq 0$ when $C_0(F) =0$ follows from Proposition \ref{prop:summary}(1).  Therefore, we obtain a slope-like function
\[
  \mu^\ast (-) := \frac{C_1(-)}{C_0(-)} = \frac{\ch_{01}(-)\cdot H\cdot D}{\rk (-)}
\]
on $\Coh (X)$.
\item On the abelian category $\mathcal A := \{\Coh^{\leq 0}\}^\uparrow$, we can define the functions
\begin{align*}
  C_0 (-) &:= \ch_{10}(-)\cdot D^2, \\
  C_1 (-) &:= \ch_{11}(-)\cdot D.
\end{align*}
That $C_0$ is nonnegative on $\mathcal A$ follows from Proposition \ref{prop:summary}(1).  On the other hand, for any $F \in \mathcal A$, we know $F$ is $\Phi$-WIT$_0$ from \cite[Remark 3.14]{Lo11}. And if $C_0(F)=0$, we have $C_1 (F) \geq 0$ by Remark \ref{remark1}.  We thus obtain the following  slope-like function
\[
  \mu_\ast (-) := \frac{C_1(-)}{C_0(-)} =  \frac{\ch_{11}(-)\cdot D}{\ch_{10}(-)\cdot D^2}
\]
on $\mathcal A = \{\Coh^{\leq 0}\}^\uparrow$.
\end{itemize}

\subsection{Results on preservation of semistability}\label{sec-preservation}

We now consider the preservation of Gieseker semistability in the rest of the discussion.  Recall that we are writing $\omega := H + nD$.

The first of our theorems on preservation of semistability is the following:

\begin{theorem}\label{0602-theorem1}
Let $\ch$ be a fixed Chern character of the form $\ch = \begin{pmatrix} 0 & 0 & 0 \\ \ast & \ast & \ast \end{pmatrix}$ where $\ch_{10} \neq 0$. Then we have an isomorphism between the following moduli spaces:
\begin{itemize}
\item[(a)] The moduli space of $\omega$-semistable sheaves $F$ on $X$ with $\ch (F) = \ch$;
\item[(b)] The moduli space of $\omega$-semistable sheaves $E$ for $n \gg 0$ on $X$ with $\ch (E) = \Phi^{H} (\ch)$.
\end{itemize}
\end{theorem}

Notice that the moduli space in (a) is independent of the value of $n$, because for the given Chern character $\ch$, $H$ does not play a role in $\omega$-semistability. The only relevant intersection numbers are $\ch_{10} \cdot D^2$, $\ch_{11} \cdot D$ and $\ch_{12}$.

On an elliptic surface $X$, Yoshioka established an isomorphism between a moduli of 1-dimensional $\omega$-semistable sheaves and a moduli of 2-dimensional (and torsion-free) $\omega$-semistable sheaves on $X$  \cite[Theorem 3.15]{YosAS} (see also \cite[Proposition 3.4.5]{YosPII}).  Also, on elliptic Calabi-Yau threefolds $X$, one can find similar results for slope semistable sheaves in \cite[Lemma 6.64]{FMNT}.

We break the proof of Theorem \ref{0602-theorem1} into Lemma \ref{0602-lemma2} and Lemma \ref{0602-lemma3}.

\begin{lemma}\label{0602-lemma2}
Suppose $F \in \coh (X)$ is a 2-dimensional sheaf that is $\omega$-semistable with  $\ch (F) = \begin{pmatrix} 0 & 0 & 0 \\ \ast & \ast & \ast \end{pmatrix}$.  Then $\wh{F}$ is $\omega$-semistable for $n \gg 0$.
\end{lemma}

\begin{proof}
To begin with, that $\ch(F)$ is of the form described implies $F \in \{\coh^{\leq 0}\}^\uparrow$ by Proposition \ref{prop:summary}(3) above.  Hence $F$ is $\Phi$-WIT$_0$ \cite[Remark 3.14]{Lo11}. The $\omega$-semistability of $F$ means that it is pure of dimension 2.  Hence $\wh{F}$ is pure of dimension 3 by \cite[Lemma 4.3]{Lo11}.

Take any short exact sequence
\[
0 \to A \to \wh{F} \to B \to 0
\]
in $\coh (X)$ in which $A, B \neq 0$.  This gives the exact sequence of cohomology in $\coh (X)$
\begin{equation}\label{0602-eq12}
0 \to \Phi^0 B \to \wh{A} \overset{\alpha}{\to} F \to \Phi^1 B \to 0
\end{equation}
in which both $\image \alpha$ and $\Phi^1 B$ lie in $\{\coh^{\leq 0}\}^\uparrow$.

 Since $A$ is a subsheaf of $\wh{F}$, it is $\Phi$-WIT$_1$, which implies $c_1(A)\cdot D^2 \leq 0$ by Remark \ref{remark1}.  If $c_1(A)\cdot D^2 <0$, then we have
\begin{equation}\label{0602-eq11}
  \frac{c_1(A)\cdot D^2}{\rk (A)} < 0 = \frac{c_1(\wh{F})\cdot D^2}{\rk (\wh{F})},
\end{equation}
in which case $A$ does not destabilise $\wh{F}$ with respect to $\omega$-semistability for $n \gg 0$.  Therefore, we assume $c_1(A)\cdot D^2 =0=c_1(\wh{F})\cdot D^2$ from now on, i.e.\ $\ch (A) = \begin{pmatrix} \ast & \ast & \ast \\ 0 & \ast & \ast \end{pmatrix}$.

That $c_1(A)\cdot D^2 =0$ implies $\rk (\wh{A})=0$, and so $\Phi^0 B$ is a $\Phi$-WIT$_1$ torsion sheaf, implying $\Phi^0 B \in \coh (\pi)_{\leq 1}$ by \cite[Lemma 2.6]{Lo7}.  Hence $\ch_2 (\Phi^0 B)\cdot D \leq 0$ by Lemma \ref{lem:pos0} together with Remark \ref{remark1}, giving us $\ch_2 (\wh{A}) \cdot D \leq \ch_2 (\image \alpha) \cdot D$.  We also have $c_1(\wh{A})\cdot D^2 = c_1 (\image \alpha) \cdot D^2$.  Therefore, we have $\mu_\ast (\wh{A})\leq \mu_\ast (\image \alpha)$.  Overall, we now have
\begin{align*}
  \mu^\ast (A) &= \mu_\ast (\wh{A}) \\
  &\leq \mu_\ast (\image \alpha) \\
  &\leq \mu_\ast (F) \text{ by the $\mu_\ast$-semistability of $F$ and Lemma \ref{0602-lemma4}} \\
  &= \mu^\ast (\wh{F}).
\end{align*}

Now, if $\mu^\ast (A) < \mu^\ast (\wh{F})$, then $A$ would not destabilise $\wh{F}$ with respect to $\omega$-semistability for $n \gg 0$.  Hence we will also assume that $\mu^\ast (A) = \mu^\ast (\wh{F})$ from now on.

Note that all the terms in the exact sequence \eqref{0602-eq12} are torsion sheaves.  In particular, we have
\begin{align*}
  -\ch_2 (A)\cdot D &= c_1 (\wh{A})\cdot H \cdot D \\
  &\geq 0 \text{ by Proposition \ref{prop:summary}(1)},
\end{align*}
 and so $\ch_2 (A)\cdot D \leq 0$.  If we have $\ch_2 (A) \cdot D < 0=\ch_2 (\wh{F}) \cdot D$, then $A$ would not destabilise $\wh{F}$ with respect to $\omega$-semistability for $n \gg 0$.  Thus we will also assume that $\ch_2 (A) \cdot D = 0$ from this point onwards, and we have reduced to the case where $\ch (A) = \begin{pmatrix} \ast & \ast & \ast \\ 0 & 0 & \ast \end{pmatrix}$, i.e.\ $\ch (\wh{A}) =  \begin{pmatrix} 0  & 0 & \ast \\ \ast & \ast & \ast \end{pmatrix}$ and so $\wh{A} \in \coh^{\lrcorner}(X)$.  It follows that $\Phi^0 B$ also lies in $\coh^{\lrcorner}(X)$ by Lemma \ref{0602-lemma4}.  Since $\Phi^0B$ is $\Phi$-WIT$_1$ and torsion, it lies in  $\coh (\pi)_{\leq 1}$ by \cite[Lemma 2.6]{Lo7}, which  means $c_1(\Phi^0 B)\cdot D^2=0$.  Hence $\ch (\Phi^0 B) = \begin{pmatrix} 0 & 0 & \ast \\ 0 & \ast & \ast \end{pmatrix}$.  We thus see that $\Phi^0 B$ is a $\Phi$-WIT$_1$ sheaf lying in $\coh^{\leq 1}(X)$, which in turn implies $\Phi^0 B$ is a fiber sheaf by \cite[Remark 3.24]{Lo11}.  Hence $\ch (\Phi^0 B) = \begin{pmatrix} 0 & 0 & \ast \\ 0 & 0 & b \end{pmatrix}$, where $b \leq 0$ by Remark \ref{remark1}.

 From the short exact sequence in $\coh (X)$
 \begin{equation}\label{0602-eq15}
 0 \to \Phi^0 B \to \wh{A} \to \image \alpha \to 0,
 \end{equation}
 we now have the short exact sequence (by applying $\Phi$)
 \begin{equation}\label{0602-eq14}
 0 \to A \to \wh{\image \alpha} \to \wh{\Phi^0 B} \to 0,
 \end{equation}
 in which $\rk (A) = \rk (\wh{\image \alpha})$ while
 \begin{align*}
   \ch_2 (A)\cdot H &= \ch_2 (\wh{\image \alpha})\cdot H - \ch_2 (\wh{\Phi^0B})\cdot H \\
   &= \ch_2 (\wh{\image \alpha})\cdot H - (-\ch_{12}(\Phi^0B)) \\
   &\leq \ch_2 (\wh{\image \alpha})\cdot H.
 \end{align*}
 That is, we have
 \begin{equation}\label{0602-eq16}
   \frac{\ch_2 (A)\cdot H}{\rk A} \leq \frac{\ch_2 (\wh{\image \alpha})\cdot H}{\rk (\wh{\image \alpha})}.
 \end{equation}

 Since $\wh{\Phi^0 B}$ is a fiber sheaf, from \eqref{0602-eq14} we have $\mu^\ast (\wh{\image \alpha}) = \mu^\ast (A)=\mu^\ast (\wh{F})$, which yields $\mu_\ast (\image \alpha) = \mu_\ast (\wh{A}) = \mu_\ast (F)$.  Now, Lemma \ref{0602-lemma4} says that $\image (\alpha), \Phi^1B$ both have Chern characters of the form $\begin{pmatrix} 0 & 0 & 0 \\ \ast & \ast & \ast \end{pmatrix}$. The $\omega$-semistability of $F$ now implies
 \[
   \frac{\chi (\image \alpha)}{c_1 (\image \alpha)\cdot D^2} \leq \frac{\chi (F)}{c_1(F)\cdot D^2},
 \]
 which is equivalent to
 \[
   \frac{\ch_2 (\wh{\image \alpha})\cdot H}{\rk (\wh{\image \alpha})} \leq \frac{\ch_2 (\wh{F})\cdot H}{\rk (\wh{F})}.
 \]
 This inequality, together with \eqref{0602-eq16}, gives
 \[
    \frac{\ch_2 (A)\cdot H}{\rk A} \leq \frac{\ch_2 (\wh{F})\cdot H}{\rk (\wh{F})}.
 \]
 If we have strict inequality here, then $A$ would not destabilise $\wh{F}$ with respect to $\omega$-semistability with respect to $n \gg 0$; therefore, we will assume that we have equality here.

 Now, since $\wh{A} \in \coh^{\lrcorner}$, we have $\ch_{02}(\wh{A})\cdot H \geq 0$ by Proposition \ref{prop:summary}(2).  Hence $\chi (A) = -\ch_{02}(\wh{A})\cdot H \leq 0$.  
 It follows that
 \[
   \frac{\chi (A)}{\rk A} \leq 0 = \frac{\chi (\wh{F})}{\rk \wh{F}}.
 \]
 This concludes the proof that $\wh{F}$ is $\omega$-semistable for $n \gg 0$.
%
\end{proof}

Now we prove the converse of Lemma \ref{0602-lemma2}:

\begin{lemma}\label{0602-lemma3}
Suppose $E \in \coh (X)$ is a  coherent sheaf supported in dimension 3 that is $\omega$-semistable for $n \gg 0$ with $\ch (E) = \begin{pmatrix} \ast & \ast & \ast \\ 0 & 0 & 0 \end{pmatrix}$.  Then $E$ is torsion-free, $\Phi$-WIT$_1$ and $\wh{E}$ is $\omega$-semistable.
\end{lemma}

\begin{proof}
That $E$ is $\omega$-semistable for $n \gg 0$ and that it is supported in dimension 3 together imply it is torsion-free.  The $\omega$-semistability for $n \gg 0$ also implies $E$ is $\mu_f$-semistable.  Since $\mu_f (E)=0$ by assumption, Lemma \ref{lemma10} tells us that  $E$ is $\Phi$-WIT$_1$ and $\wh{E}$ is a torsion sheaf.  Moreover, since $\ch(\wh{E}) = \begin{pmatrix} 0 & 0 & 0 \\ \ast & \ast & \ast \end{pmatrix}$, we have $\wh{E} \in \{\coh^{\leq 0}\}^\uparrow$ by Proposition \ref{prop:summary}(3) above.  We will now check that $\wh{E}$ is indeed $\omega$-semistable.

Take any short exact sequence
\begin{equation}\label{0602-eq17}
0 \to M \to \wh{E} \to N \to 0
\end{equation}
in $\coh (X)$.  Then the Chern characters of $M, N$ are also of the form $\begin{pmatrix} 0 & 0 & 0 \\ \ast & \ast & \ast \end{pmatrix}$ by Lemma \ref{0602-lemma4}, and they both lie in  $\{\coh^{\leq 0}\}^\uparrow$ by Proposition \ref{prop:summary}(3) above.  That is, all the terms in \eqref{0602-eq17} are $\Phi$-WIT$_0$, and \eqref{0602-eq17} is taken by $\Phi$ to the short exact sequence in $\coh (X)$
  \[
0 \to \wh{M} \to E \to \wh{N} \to 0.
  \]
If $c_1(M) \cdot D^2=0$, then $M \in \coh (\pi)_{\leq 1}$, and so $\wh{M}$ is a subsheaf of $E$ lying in $\coh (\pi)_{\leq 1} \subset \coh^{\leq 2}(X)$, contradicting the purity of $E$.  Therefore, we can assume that $c_1(M) \cdot D^2 >0$.  Note that $\ch (\wh{M})$ is also of the form $\begin{pmatrix} \ast & \ast & \ast \\ 0 & 0 & 0 \end{pmatrix}$.

Now, if we have
\[
  \mu_\ast (M) := \frac{\ch_2(M)\cdot D}{\ch_1(M)\cdot D^2} > \frac{\ch_2(\wh{E})\cdot D}{\ch_1(\wh{E})\cdot D^2}=: \mu_\ast (\wh{E}),
\]
then we have, equivalently, $\mu^\ast (\wh{M}) > \mu^\ast (E)$, contradicting that $E$ is $\omega$-semistable for $n \gg 0$.  Hence we have $\mu_\ast (M) \leq \mu_\ast (\wh{E})$.  (That is,  $\wh{E}$ is at least $\mu_\ast$-semistable.)

Suppose $\mu_\ast (M) = \mu_\ast (\wh{E})$ but $\frac{\chi(M)}{\ch_1(M)\cdot D^2} > \frac{\chi (\wh{E})}{\ch_1 (\wh{E})\cdot D^2}$.  This inequality is equivalent to $\frac{\ch_2 (\wh{M})\cdot H}{\rk \wh{M}} > \frac{\ch_2(E) \cdot H}{\rk E}$, which implies that $\wh{M}$ destabilises $E$ with respect to $\omega$-semistability for $n \gg 0$.  Hence we must have $\frac{\chi(M)}{\ch_1(M)\cdot D^2} \leq  \frac{\chi (\wh{E})}{\ch_1 (\wh{E})\cdot D^2}$, concluding the proof that $\wh{E}$ is $\omega$-semistable.
\end{proof}

\begin{proof}[Proof of Theorem \ref{0602-theorem1}]
Take any coherent sheaf $F$ on $X$ that is $\omega$-semistable with $\ch(F)=\ch$.  By Lemma \ref{0602-lemma2} and its proof, we know $F$ is $\Phi$-WIT$_0$ and $\wh{F}$ is $\omega$-semistable for $n \gg 0$.  Conversely, take any coherent sheaf $E$ on $X$ with $\ch (E) = \Phi^{H}(\ch)$.  Then by Lemma \ref{0602-lemma3}, we know $E$ is $\Phi$-WIT$_1$ and $\wh{E}$ is $\omega$-semistable.  Thus the theorem is proved.
\end{proof}

The following  is a variant of Theorem \ref{0602-theorem1}:

\begin{theorem}\label{0602-theorem2}
Let $\omega = H+nD$.  We have an equivalence of categories induced by $\Phi$
\begin{multline}
\{ E \in \{\Coh^{\leq 0}\}^\uparrow : \ch_1(E) \neq 0, E \text{ is $\mu_\ast$-semistable}\} \\
 \overset{\thicksim}{\to} \{E \in \Phi (\{\Coh^{\leq 0}\}^\uparrow) : \ch_0(E) \neq 0, E \text{ is $\mu_\omega$-semistable for $n \gg 0$}\}.
\end{multline}
\end{theorem}

In the statement of  Theorem \ref{0602-theorem2}, it is implicit that a coherent sheaf lying in the category on the left is a $\Phi$-WIT$_0$ sheaf supported in dimension 2, while a coherent sheaf lying in the category on the right is a $\Phi$-WIT$_1$ sheaf supported in dimension 3 (and torsion-free).

Also, we can interpret Theorem \ref{0602-theorem2} as follows: the notion of semistability for coherent sheaves in $\{\Coh^{\leq 0}\}^\uparrow$ that corresponds, under the FMT $\Phi$, to $\mu_\omega$-semistability for $n \gg 0$, is precisely the notion of $\mu_\ast$-semistability.


\begin{proof}
Take any coherent sheaf $E$ lying in the category on the left.  That $E$ is $\mu_\ast$-semistable implies it is pure 2-dimensional.  Also, that $E$ lies in $\{\coh^{\leq 0}\}^\uparrow$ implies $E$ is $\Phi$-WIT$_0$ by \cite[Remark 3.14]{Lo11}, and thus $\wh{E}$ is torsion-free by \cite[Lemma 4.3]{Lo11}.  Now, consider any short exact sequence in $\coh (X)$
\[
  0 \to A \to \wh{E} \to B \to 0
\]
in which $A, B \neq 0$; it yields the long exact sequence in $\coh (X)$
\[
0 \to \Phi^0 B \to \wh{A} \overset{\alpha}{\to} E \to \Phi^1 B \to 0.
\]
Since $A$ is $\Phi$-WIT$_1$, we have $\mu_f (A) \leq 0$ by Remark \ref{remark1}.  Also, that $\ch_0(E)=0$ implies $\mu_f (\wh{E})=0$, and so $\mu_f (A) \leq 0 = \mu_f (\wh{E})$.  If we have $\mu_f (A) < \mu_f (\wh{E})$, then $A$ does not destabilise $\wh{E}$ with respect to $\mu_\omega$ as $n \gg 0$.  As a result, let us assume that $\mu_f (A)=0=\mu_f(\wh{E})$.  Then $c_1(A)\cdot D^2 = 0 = c_1(\wh{E})\cdot D^2$, and so $c_1(B)\cdot D^2=0$, which implies $\rk (\Phi B)=0$.  Since we already have $\rk (\Phi^1 B)=0$ (since $E$ is torsion), we must also have $\rk (\Phi^0 B)=0$.  Hence $\Phi^0 B$ is a $\Phi$-WIT$_1$ torsion sheaf, and hence lies in $\coh (\pi)_{\leq 1}$ by \cite[Lemma 2.6]{Lo7}.

Note that $c_1(\image \alpha)\cdot D^2$ must be nonzero, for if it is zero, then $c_1(\wh{A})\cdot D^2=0$ as well (since $\Phi^0B \in \coh (\pi)_{\leq 1}$), which is equivalent to $\rk (A)=0$, which is impossible since $\wh{E}$ is torsion-free and $A \neq 0$.  Thus $\mu_\ast (\image \alpha) < \infty$ and the $\mu_\ast$-semistability of $E$ gives $\mu_\ast (\image \alpha) \leq \mu_\ast (E)$.  Then
\begin{align*}
  \mu^\ast (A) = \frac{\ch_1(A)\cdot HD}{\rk (A)} &= \frac{\ch_2(\wh{A})\cdot D}{\ch_1 (\wh{A})\cdot D^2} \\
  &= \frac{ (\ch_2(\Phi^0B)+\ch_2(\image \alpha))\cdot D}{\ch_1(\wh{A})\cdot D^2} \\
  &\leq \frac{\ch_2 (\image \alpha)}{\ch_1 (\image \alpha)\cdot D^2} \\
  &= \mu_\ast (\image \alpha) \\
  &\leq \mu_\ast (E) = \mu^\ast (\wh{E}).
\end{align*}
Above, the first inequality holds because $\ch_2 (\Phi^0 B)\cdot D \leq 0$ (from $\Phi^0B$ being $\Phi$-WIT$_1$ and Remark \ref{remark1}) and $\ch_1(\wh{A})\cdot D^2 =\ch_1 (\image \alpha)\cdot D^2$ (from $\Phi^0 B \in \coh (\pi)_{\leq 1}$).  This shows that $\wh{E}$ is $\mu_\omega$-semistable for $n \gg 0$.  Hence $\wh{E}$ lies in the category on the right.

For the other direction, take any coherent sheaf $E$ lying in the category on the right.  That $E$ is $\mu_\omega$-semistable for $n \gg 0$ implies $E$ is torsion-free and is $\mu_f$-semistable.  Also, the assumption $E \in \Phi (\{\coh^{\leq 0}\}^\uparrow)$ implies that $E$ is $\Phi$-WIT$_1$ and $c_1(E)\cdot D^2 = 0$.

Consider any short exact sequence in $\coh (X)$
\[
0 \to M \to \wh{E} \to N \to 0
\]
where $M, N \neq 0$.  Then $M, N$ both lie in $\{\coh^{\leq 0}\}^\uparrow$ and are both $\Phi$-WIT$_0$.  This short exact sequence is taken by $\Phi$ to the short exact sequence in $\coh (X)$
\[
0 \to \wh{M} \to E \to \wh{N} \to 0
\]
in which all the terms satisfy $c_1(-)\cdot D^2=0$.  Since $E$ is torsion-free and $\wh{M}\neq 0$, we have $\rk(\wh{M})>0$.  We also have  $\mu_f (\wh{M})=0=\mu_f (E)$.  The $\mu_\omega$-semistability of $E$ for $n \gg 0$ now implies
\[
\mu_\ast (M) = \mu^\ast (\wh{M})\leq \mu^\ast (E) = \mu_\ast (\wh{E})
\]
and so $\wh{E}$ is $\mu_\ast$-semistable, and lies in the category on the left.
\end{proof}


\begin{remark}\label{0602-remark4}
From \eqref{0602-eq7}, we have the following relations for coherent sheaves $F$ on $X$ supported in dimension 2 (here $\omega := H + nD$):
\begin{itemize}
\item if $F$ is $\mu_\omega$-semistable for $n \gg 0$, then $F$ is $\mu_\ast$-semistable;
\item if $F$ is $\mu_\ast$-stable, then $F$ is $\mu_\omega$-stable for $n \gg 0$.
\end{itemize}
\end{remark}

\appendix
\section{Preservation of semistability on elliptic surfaces}\label{sec-appendix}

In this section, we consider the trivial elliptic surface $X = C \times T$ where $C$ is an elliptic curve, and $T$ is an arbitrary smooth projective curve.  We regard $X$ as an elliptic surface via the second projection $\pi : X = C \times T \to T$.  Using the techniques we have developed, we recover Yoshioka's result on preservation on Gieseker semistabiilty (i.e.\ \cite[Theorem 3.15]{YosAS}; see also \cite[Proposition 3.4.5]{YosPII}), in the case of the trivial elliptic surface $X$.

As a reminder, we consider the Fourier-Mukai transform $\Phi : X \to X$ as defined in Section \ref{sec-FMTcohom}, so that if $E \in D(X)$ has $\ch(E) = \begin{pmatrix} a_{00} & a_{01} \\ a_{10} & a_{11} \end{pmatrix}$, then we have $\ch(\Phi (E)) = \begin{pmatrix} a_{10} & a_{11} \\ -a_{00} & -a_{01} \end{pmatrix}$.

\subsection{Semistability with respect to fiber-like polarisations}

We use $f$ for the class of a fiber $C \times \{ \text{point} \}$ and $h$ for the class of a horizontal section $\{ \text{point} \} \times T$ of the fibration $\pi$.
Consider  the polarisation $\omega = th + sf$ where $t,s >0$. Given an object $E \in D(X)$, the Hilbert polynomial of $E$ with respect to the polarisation $\omega$ is 
\[
  P_\omega (E,m) = m^2 a_{00} ts + m(a_{01}t + a_{10}s) + a_{11}.
\]
Setting $t=1$, we can arrange the coefficients of $P_\omega (E,m)$  in an array as follows:

\begin{center}
\begin{tabular}{c | c c }
 & $s$ & $s^0$ \\
 \hline
 $m^2$ & $a_{00}$ &  \\
 $m$ & $a_{10}$ & $a_{01}$ \\
 $m^0$ & & $a_{11}$
\end{tabular}
\end{center}


From this table, we see that in determining the  semistabilty of a coherent sheaf $E$ on $X$ with respect to $\omega$ for $s \gg 0$, we can simply  compare the following quantities, in the listed order:
\begin{itemize}
\item when $a_{00} \neq 0$, i.e.\ when $E$ is supported in dimension 2:
\begin{equation}\label{seq2-2}
  \frac{a_{10}}{a_{00}} \to \frac{a_{01}}{a_{00}} \to \frac{a_{11}}{a_{00}}.
\end{equation}
(Note that $\frac{a_{10}}{a_{00}}, \frac{a_{01}}{a_{00}}$ are constant multiples of $\mu_f(E), \mu^\ast (E)$, respectively.)
\item when $a_{00} =0$, but $a_{10} \neq 0$ or $a_{01} \neq 0$, i.e.\ when $E$ is supported in dimension 1:
\begin{equation}\label{seq2-1}
  \frac{a_{11}}{a_{10}} \to \frac{a_{11}}{a_{01}}.
\end{equation}
\end{itemize}

This can be stated more precisely in the following manner:

\begin{lemma}\cite[Lemma 3.4.1]{YosPII}\label{lemma20}
\begin{itemize}
\item[(1)] Let $E$ be a torsion-free coherent sheaf on $X$.  Then $E$ is $\omega$-semistable for $s \gg 0$ iff for every proper subsheaf $E'$ of $E$, one of the following conditions holds:
    \begin{itemize}
    \item[(a)]
    \begin{equation}\label{eq-Yos3.31}
    \frac{c_1(E)\cdot f}{\rank (E)} > \frac{c_1(E')\cdot f}{\rank E'},
    \end{equation}
    \item[(b)]
    \begin{equation}\label{eq-Yos3.32}
    \frac{c_1(E)\cdot f}{\rank E} = \frac{c_1 (E')\cdot f}{\rank E'}, \frac{c_1(E)\cdot h}{\rank E} > \frac{c_1(E')\cdot h}{\rank E'},
    \end{equation}
    \item[(c)]
    \begin{equation}\label{eq-Yos3.33}
    \frac{c_1(E)\cdot f}{\rank E} = \frac{c_1(E')\cdot f}{\rank E'}, \frac{c_1(E)\cdot h}{\rank E} = \frac{c_1(E')\cdot h}{\rank E'}, \frac{\chi (E)}{\rank E} \geq \frac{\chi (E')}{\rank E'}.
    \end{equation}
    \end{itemize}
\item[(2)] Let $F$ be a 1-dimensional coherent sheaf on $Y$ with $c_1(F)\cdot f \neq 0$.  Then $F$ is $\omega$-semistable for $s \gg 0$ iff for every proper subsheaf $F'$ of $F$, one of the following conditions holds:
    \begin{itemize}
    \item[(a)]
    \begin{equation}\label{eq-Yos3.34}
    c_1(F')\cdot f \cdot \frac{\chi (F)}{c_1(F)\cdot f} > \chi (F') ,
    \end{equation}
    \item[(b)]
    \begin{equation}\label{eq-Yos3.35}
    c_1(F')\cdot f \cdot \frac{\chi (F)}{c_1(F)\cdot f} = \chi (F'), c_1(F')\cdot h \cdot \frac{\chi (F)}{c_1(F)\cdot h} \geq \chi (F').
    \end{equation}
    \end{itemize}
\end{itemize}
\end{lemma}

\subsection{Preservation of Gieseker stability on elliptic surfaces (Yoshioka)}

In this section, we prove the following:

\begin{theorem}\label{theorem3}
Let $X = C\times T$ where $C$ is an elliptic curve and $T$ a smooth projective curve.  Consider $X$ as an elliptic surface via the second projection $\pi : X = C \times T \to T$, and let $\Phi : D(X) \overset{\thicksim}{\to} D(X)$ be as above.  Let $\ch$ be a fixed Chern character such that it is the Chern character of a 1-dimensional sheaf, with $c_1 \cdot f>0$ and $\chi=\ch_2>0$.  Then we have an equivalence of categories
\begin{multline*}
  \{F \in \Coh (X) : \ch(F) = \ch, F \text{ is $\omega$-semistable for $s \gg 0$} \} \overset{\Phi}{\to} \\
  \{ E \in \Coh (X): \ch(E)=\Phi^H(\ch), E \text{ is $\omega$-semistable for $s \gg 0$}\}.
\end{multline*}
\end{theorem}

We divide the proof of Theorem \ref{theorem3} into Lemmas \ref{pro3} and \ref{pro4}.

\begin{lemma}\label{pro3}
Suppose $E \in \Coh (X)$  satisfies:
\begin{itemize}
\item[(a)] $E$ is $\omega$-semistable  for $s \gg 0$;
\item[(b)] $\ch(\Phi (E)[1])$ is the Chern character of some coherent sheaf supported in dimension 1, with $\chi (\Phi (E)[1]) > 0$.
\end{itemize}
Then $E$ lies in $\mathcal F_X \cap \Phi (W_{0,X} \cap \Coh^{\leq 1}(X))$ and is  $\Phi$-WIT$_1$, and $\widehat{E}$ is $\omega$-semistable  for $s \gg 0$.  Also, we have $c_1(\widehat{E})\cdot f >0$.
\end{lemma}

Note that, since the Chern character is additive on exact triangles, we can extend the definitions of the functions $\mu_f$ and $\mu^\ast$ so that they are defined on all of $D^b(X)$.

\begin{proof}
Take any $E \in \Coh (X)$ satisfying properties (a) and (b).  Note that, for any 1-dimensional sheaf $T$ on $X$, the restriction $T|_s$ is a 0-dimensional sheaf for a general $s \in C$, and so $\mu_f (\Phi (T))=0$.  Thus property (b) implies that $\mu_f (E)=0$.  By \eqref{seq2-2}, that $E$ is  semistable with respect to $\omega$ for $s \gg 0$  implies that it is $\mu_f$-semistable (and torsion-free).  Therefore, by Lemma \ref{lemma10}, we have $E \in \mathcal F_X \cap \Phi (W_{0,X} \cap \Coh^{\leq 1}(X))$.  In particular, we know that $E$ is $\Phi$-WIT$_1$.  Since $\widehat{E}$ is supported in dimension 1 but is not a fiber sheaf, we have $c_1(\widehat{E})\cdot f >0$.  Besides, $\wh{E}$ is necessarily pure 1-dimensional, for any 0-dimensional subsheaf $E_0 \subseteq \wh{E}$ would be taken by $\Phi$ to a fiber subsheaf of $E$, contradicting $E$ being torsion-free.

Now, take any nonzero proper subsheaf $F \subset \wh{E}$.  Then $F$ is also pure 1-dimensional.  The short exact sequence in $\Coh (Y)$
\[
0 \to F \to \wh{E} \to \wh{E}/F \to 0
\]
yields the exact sequence of cohomology in $\Coh (X)$
\[
  0 \to \Phi^0 F \to E \overset{\alpha}{\to} M \to \Phi^1 F \to 0
\]
where $M := \wh{(\wh{E}/F)}$.

If $\Phi^0 F=0$, then $F$ is $\Phi$-WIT$_1$, and is necessarily a fiber sheaf \cite[Lemma 2.6]{Lo7}.  Thus $c_1(F)\cdot f=0$, $c_1(F) \cdot h >0$, and $\chi (F) \leq 0$.  Then  the assumption that $\chi (\wh{E})>0$ ensures $F$ does not destabilise $\wh{E}$ for $s \gg 0$.  From now on, let us suppose $\Phi^0 F\neq 0$.  Then $\Phi^0 F$ is torsion-free.




Since $F$ is a 1-dimensional sheaf, we have $c_1 (\Phi^0 F)\cdot f =0$.  We also have $c_1 (E)\cdot f =0$, and so the semistability of $E$ with respect to $\omega$ for $s \gg 0$ gives us $\mu^\ast (\Phi^0 F) \leq \mu^\ast (E)$.  On the other hand, since the $\Phi$-WIT$_1$ part of $F$ is a fiber sheaf \cite[Lemma 2.6]{Lo7}, we see that $\Phi^1 F$ is a fiber sheaf, and so $\rank (\Phi^1 F)=0$.  Hence
\begin{align}
  c_1 (\Phi F) \cdot h &= c_1 (\Phi^0 F)\cdot h - c_1(\Phi^1 F)\cdot h \notag \\
  &\leq c_1 (\Phi^0 F) \cdot h, \label{eq49}
\end{align}
with equality iff $\Phi^1 F$ is 0-dimensional.

Suppose we have strict inequality in \eqref{eq49}.  Then, noting that $\rank (\Phi F)= \rank (\Phi^0 F)$ since $\Phi^1 F$ is a fiber sheaf, we have
\begin{align*}
  \mu^\ast (\Phi F) = \frac{c_1(\Phi F)}{\rank (\Phi F)}   < \frac{c_1(\Phi^0F)\cdot h}{\rank (\Phi^0 F)}   &= \mu^\ast (\Phi^0 F) \\
  &< \mu^\ast (E).
\end{align*}
On the other hand, from the formula for the cohomological Fourier-Mukai transform in Proposition \ref{prop:product_change_class}, we have
\[
  c_1(\wh{E})\cdot f = \rank (E) \text{\quad and \quad} \chi (\wh{E}) = c_1(E) \cdot h.
\]
It follows that
\[
  \frac{\chi (\wh{E})}{c_1(\wh{E})\cdot f} = \mu^\ast (E),
\]
and similarly for $F$.  Thus $\mu^\ast (\Phi F) < \mu^\ast (E)$ is equivalent to $\frac{\chi (F)}{c_1(F)\cdot f} < \frac{\chi (\wh{E})}{c_1(\wh{E})\cdot f}$, and so  $F$ does not destabilise $\wh{E}$ with respect to $\omega$ for $s \gg 0$.

To finish the proof, consider the case when we have equality in \eqref{eq49}.  In this case, $\Phi^1 F$ is 0-dimensional, and $c_1 (\Phi F)\cdot h = c_1 (\Phi^0 F) \cdot h$.  It follows that
\begin{equation}\label{eq50}
  \mu^\ast (\Phi F) = \mu^\ast (\Phi^0 F) \leq \mu^\ast (E),
\end{equation}
where the second inequality is from above.  The inequality \eqref{eq50} corresponds to
\begin{equation}\label{eq51}
  \frac{\chi (F)}{c_1(F)\cdot f} \leq \frac{\chi (\wh{E})}{c_1 (\wh{E})\cdot f}.
\end{equation}
If we have strict inequality in \eqref{eq51}, then we are done.  So suppose we have equality in \eqref{eq51}.  Our aim now is to show
\begin{equation}\label{eq52}
  c_1(F) \cdot h \cdot \frac{\chi (\wh{E})}{c_1(\wh{E})\cdot h} \geq \chi (F).
\end{equation}
Since $c_1(F) \cdot h \geq 0$,  we can assume $c_1(\wh{E})\cdot h >0$.

We now have
\begin{align}
   c_1(F) \cdot h  \cdot \frac{\chi (\wh{E})}{c_1(\wh{E})\cdot h}  &=  c_1(F) \cdot h \cdot \frac{\chi (\wh{E})}{c_1(\wh{E})\cdot f} \cdot \frac{ c_1(\wh{E})\cdot f}{c_1(\wh{E})\cdot h}  \notag\\
  &=  c_1(F) \cdot h \cdot \frac{\chi(F)}{c_1(F)\cdot f} \cdot \frac{c_1(\wh{E})\cdot f}{c_1(\wh{E})\cdot h} \text{ since we have equality in \eqref{eq51}}. \label{eq53}
\end{align}
Note that, if we consider the short exact sequence $0 \to F_0 \to F \to F_1 \to 0$ in $\Coh(Y)$ where $F_i$ is $\Phi$-WIT$_i$, then
\begin{align}
  c_1(F)\cdot f &= c_1(F_0)\cdot f \text{ and} \label{eq55}\\
  c_1(F) \cdot h &= c_1 (F_0)\cdot h + c_1(F_1)\cdot h \\
  &\geq c_1(F_0)\cdot h. \label{eq56}
\end{align}
Combining this with the semistability of $E$ for $s \gg 0$, and noting we already have $c_1(\Phi^0 F)\cdot f = 0 = c_1(E)\cdot f$ as well as $\mu^\ast (\Phi^0 F)=\mu^\ast (E)$, we obtain
\begin{equation*}
- \frac{c_1(\wh{E})\cdot h}{c_1(\wh{E})\cdot f} = \frac{\chi (E)}{\rank (E)} \geq \frac{\chi (\Phi^0 F)}{\rank (\Phi^0 F)} = - \frac{ c_1(F_0)\cdot h}{c_1(F_0)\cdot f},
\end{equation*}
which in turn gives
\begin{align*}
 c_1(F) \cdot h \cdot \frac{\chi(F)}{c_1(F)\cdot f} \cdot \frac{c_1(\wh{E})\cdot f}{c_1(\wh{E})\cdot h}  &\geq  c_1(F) \cdot h \cdot \frac{\chi (F)}{c_1(F)\cdot f} \cdot \frac{c_1(F_0)\cdot f}{c_1(F_0)\cdot h} \\
&\geq \chi (F) \text{ by \eqref{eq55}, \eqref{eq56}}.
\end{align*}
Above, we used the fact that $c_1(G)\cdot f, c_1(G)\cdot h \geq 0$ for any 1-dimensional sheaf $G$, since $f, h$ are effective divisors on $X$.  Also, $\chi (F) >0$ by the hypothesis $\chi (\wh{E})>0$ and the equality in \eqref{eq51}.  Along with \eqref{eq53}, we now have \eqref{eq52}.
\end{proof}

\begin{lemma}\label{pro4}
Let $F \in \Coh (Y)$ be a 1-dimensional sheaf that is $\omega$-semistable for $s \gg 0$ with $c_1(F)\cdot f >0$ and $\chi (F) >0$.   Then $F$  is $\Phi$-WIT$_0$, and $\widehat{F}$ is torsion-free and $\omega$-semistable for $s \gg 0$.
\end{lemma}


\begin{proof}
First, we show that $F$ is $\Phi$-WIT$_0$: consider the short exact sequence in $\Coh (Y)$
\[
0 \to F_0 \to F \to F_1 \to 0
\]
where $F_i$ is $\Phi$-WIT$_i$.  Since $F_1$ is $\Phi$-WIT$_1$ and torsion, it is a fiber sheaf \cite[Lemma 2.6]{Lo7}.  Thus $c_1(F_1)\cdot f =0$ and $\chi (F_1) \leq 0$ by Remark \ref{remark1}, which imply $c_1(F_0)\cdot f = c_1(F)\cdot f>0$ and $\chi (F_0) \geq \chi (F)>0$.  By the semistability of $F$, it follows that $\chi (F_0)=\chi (F)$.   If $F_1 \neq 0$, then $c_1 (F_1)\cdot h >0$ nd $c_1(F_0)\cdot h < c_1(F)\cdot h$, violating \eqref{eq-Yos3.35} and hence the semistability of $F$.  Thus $F_1=0$ and $F$ is $\Phi$-WIT$_0$.

Next, we show that $\wh{F}$ is torsion-free.  Consider any torsion subsheaf $A$ of $\wh{F}$.  Then $A$ is $\Phi$-WIT$_1$, and so is a fiber sheaf \cite[Lemma 2.6]{Lo7}.  Then the injection $A \hookrightarrow \wh{F}$ is taken by $\Phi$ to a morphism $\beta : \wh{A} \to F$ where $\image \beta$ is a $\Phi$-WIT$_0$ fiber subsheaf of $F$ with $\chi (\image \beta)>0$, which violates \eqref{eq-Yos3.34}.  Thus $A$ must be zero, and $\wh{F}$ is torsion-free.

Now, we show that $\wh{F}$ is $\omega$-semistable for $s \gg 0$.  Consider any short exact sequence
\[
0 \to A \to \wh{F} \to B \to 0
\]
in $\Coh (X)$ with $0\neq A \subsetneq \wh{F}$, which yields the exact sequence in $\Coh (Y)$
\[
0 \to \Phi^0 B \to \wh{A} \overset{\alpha}{\to} F \to \Phi^1 B \to 0.
\]

Since $A$ is $\Phi$-WIT$_1$ and torsion-free, we have $c_1(A)\cdot f \leq 0$ by \cite[Lemma 2.5]{Lo7}.  If $c_1(A) \cdot f <0$, then $A$ does not destabilise $\wh{F}$ (since $c_1(\wh{F})\cdot f =0$, which follows from $F$ being torsion).  So let us assume $c_1(A)\cdot f =0$ from now on.  That $c_1(A)\cdot f=0$ then implies $\rank (\wh{A})=0$ by Proposition \ref{prop:product_change_class}, i.e.\ $\wh{A}$ is torsion.  Then $\Phi^0 B$ is $\Phi$-WIT$_1$ and torsion, and so is a fiber sheaf by \cite[Lemma 2.6]{Lo7}.   Now, if $\alpha =0$, then $\wh{A} \cong \Phi^0B$ is both $\Phi$-WIT$_0$ and $\Phi$-WIT$_1$, forcing $\wh{A}=0$.  So let us assume that $\alpha \neq 0$, i.e.\ $\image \alpha \neq 0$.

If $c_1(\image \alpha) \cdot f =0$, then $\image \alpha$ must be a fiber sheaf (since it is torsion) and $\Phi$-WIT$_0$ (being a quotient of $\wh{A}$).  This implies that $\wh{F}$ has a fiber subsheaf, contradicting its being torsion-free.  By \cite[Lemma 2.5]{Lo7}, we have $c_1 (\image \alpha)\cdot f \geq 0$, and so  we must have $c_1 (\image \alpha)\cdot f > 0$.  Now we have, noting that $\Phi^0B$ is a $\Phi$-WIT$_1$ fiber sheaf,
\begin{equation*}
  c_1 (\wh{A}) \cdot f = c_1(\image \alpha)\cdot f  >0
\end{equation*}
and
\begin{equation*}
  \chi (\wh{A}) = \chi (\image \alpha) + \chi (\Phi^0 B) \leq \chi (\image \alpha).
\end{equation*}
Therefore, we have
\begin{equation}\label{eq58}
 \frac{\chi (\wh{A})}{c_1 (\wh{A})\cdot f} \leq \frac{\chi (\image \alpha)}{c_1 (\image \alpha)\cdot f} \leq \frac{\chi (F)}{c_1(F)\cdot f},
\end{equation}
(where the second inequality follows from the semistability of $F$), which then implies
\begin{equation}\label{eq57}
  \frac{ c_1(A)\cdot h}{\rank (A)} \leq \frac{ c_1(\wh{F})\cdot h}{\rank (\wh{F})},
\end{equation}
which is \eqref{eq-Yos3.31} for $A$ and $\wh{F}$.  Assuming that equality holds in \eqref{eq57} (which also implies equality throughout \eqref{eq58}), we have
\begin{align*}
  -\frac{c_1(\wh{A})\cdot h}{c_1(\wh{A})\cdot f} &\leq  \frac{-c_1(\image \alpha)\cdot h}{c_1(\image \alpha)\cdot f} \\
  &= - c_1 (\image \alpha)\cdot h \cdot \frac{\chi (F)}{c_1(F)\cdot h}\cdot \frac{c_1(F)\cdot h}{\chi (F)} \cdot \frac{1}{c_1 (\image \alpha)\cdot f} \\
  &\leq -\chi (\image \alpha) \cdot \frac{c_1(F)\cdot h}{\chi (F)} \cdot \frac{1}{ c_1 (\image \alpha )\cdot f} \\
   &\text{\qquad since $c_1 (\image \alpha) \cdot h\cdot \frac{\chi (F)}{c_1(F)\cdot h} \geq \chi (\image\alpha)$ by the semistability of $F$} \\
   &= -\frac{\chi (F)}{c_1(F)\cdot f} \cdot \frac{c_1(F)\cdot h}{\chi (F)} \\
   &\text{\qquad by \eqref{eq58}} \\
   &= - \frac{c_1(F)\cdot h}{c_1(F)\cdot f}.
\end{align*}
Overall, we have
\[
 -\frac{c_1(\wh{A})\cdot h}{c_1(\wh{A})\cdot f} \leq  - \frac{c_1(F)\cdot h}{c_1(F)\cdot f},
\]
which corresponds to
\[
  \frac{\chi (A)}{\rank (A)} \leq \frac{\chi (\wh{F})}{\rank (\wh{F})},
\]
which is \eqref{eq-Yos3.33} for $A$ and $\wh{F}$.  We have thus shown that $\wh{F}$ is $\omega$-semistable for $s \gg 0$.
\end{proof}


\begin{thebibliography}{99}


\bibitem[ARG]{ARG}
Andreas, B. and Ruip\'{e}rez, D. Hern\'{a}ndez and G\'{o}mez, D. S\'{a}nchez,
\emph{Stable sheaves over $K3$ fibrations}.
Internat. J. Math. 21 (2010), no. 1, 25-–46.

\bibitem[BM04]{BM04}
Bridgeland, Tom and Maciocia, Antony,
\emph{Fourier-Mukai transforms for K3 and elliptic fibrations}.
J. Algebraic Geom. 11 (2002), no. 4, 629–-657.

\bibitem[BS]{BS}
M. ~Brown and I. ~Shipman,
\emph{Derived equivalences of surfaces via numerical tilting}.
Preprint, 2013. arXiv:1312.3918 [math.AG]

\bibitem[BBR]{FMNT} C. ~Bartocci, U. ~Bruzzo, D. ~ Hern\'{a}ndez-Ruip\'{e}rez,
\emph{Fourier-Mukai and Nahm Transforms in Geometry and Mathematical Physics}.
Progress in Mathematics, Vol. 276, Birkh\"{a}user, 2009.

\bibitem[BMS]{BMS} Bayer, A., Macr\`{i}, E., and Stellari, P.,
\emph{The Space of Stability Conditions on Abelian Threefolds, and on some Calabi-Yau Threefolds}, preprint.  arXiv:1410.1585 [math.AG]


\bibitem[BMT14]{BMT14}
Bayer, A., Macr\'{i}, E., and Toda, Y.,
\emph{Bridgeland Stability conditions on Threefolds I: Bogomolov-Gieseker type inequalities}.
J. Algebraic Geom. 23 (2014), 117--163.

\bibitem[BH]{BH}
Bernardara, M. and Hein, G.,
\emph{The Euclid-Fourier-Mukai algorithm for elliptic surfaces},
Asian J. Math. 18 (2014), no.2, 345--364.

\bibitem[Bri]{BriTh}
Bridgeland, T.,
\emph{Fourier-Mukai transforms for surfaces and moduli spaces of stable sheaves}.
PhD thesis, University of Edinburgh, 1998.


\bibitem[CL]{CL}
W. Y. Chuang and J. Lo,
\emph{Stability and Fourier-Mukai transforms on higher dimensional elliptic fibrations}.
Preprint.  2013.  arXiv:1307.1845 [math.AG]

\bibitem[HL10]{HL10}
Huybrechts, Daniel and Lehn, Manfred,
\emph{The geometry of moduli spaces of sheaves}.
Second edition. Cambridge Mathematical Library. Cambridge University Press, Cambridge, 2010.

\bibitem[AB10]{AB10}
Arinkin, Dmitry and Bezrukavnikov, Roman,
\emph{Perverse coherent sheaves}.
Mosc. Math. J. 10 (2010), no. 1, 3-–29, 271.

\bibitem[Dia]{Dia}
Diaconescu, D. E.,
\emph{Vertical sheaves and Fourier-Mukai transform on elliptic Calabi-Yau threefolds}.
Preprint. 2015.  arXiv:1509.07749 [math.AG]

\bibitem[GKR]{GKR}
A. L. Gorodentsev, S. A. Kuleshov and A. N. Rudakov,
\emph{t-stabilities and t-structures on triangulated categories}.
Izv. Math. 68 (2004), 749--781.

\bibitem[EGA4]{EGA4}
Grothendieck, Alexandre; Dieudonné, Jean (1965).
\emph{\'El\'ements de g\'eom\'etrie alg\'ebrique: IV. \'Etude locale des sch\'emas et des morphismes de sch\'emas, Seconde partie}.
Publications Mathématiques de l'IHÉS 24: 5–231.

\bibitem[FMW]{FMW} R. ~Friedman, J. W. ~Morgan and E. ~Witten,
\emph{Vector bundles over elliptic fibrations}.
J. Algebraic Geom., Vol. 8 (1999), 279--401.

\bibitem[Har]{Harts}
Hartshore, R.,
\emph{Algebraic Geometry}.
GTM 52.  Springer, 1977.


\bibitem[Huy06]{Huy06}
Huybrechts, D.,
\emph{Fourier-Mukai transforms in algebraic geometry}.
Oxford Mathematical Monographs. The Clarendon Press, Oxford University Press, Oxford, 2006.

\bibitem[Huy08]{Huy08}
Huybrechts, D.,
\emph{Derived and Abelian equivalences of K3 surfaces}.
J. Algebraic Geom., 17 (2008), 375--400.

\bibitem[HL]{HL}
D. ~Huybrechts and M. ~Lehn,
\emph{The Geometry of Moduli Spaces of Sheaves}.
Aspects of Mathematics, Vol. 31, Vieweg, Braunschweig, 1997.

\bibitem[HP]{HP}
G. ~Hein and D. ~Ploog,
\emph{Postnikov-stability versus semistability of sheaves}.
Asian J. Math., 18 (2014), 247--262.

\bibitem[Lo]{Lo7}
Lo, J.,
\emph{Stability and Fourier-Mukai transforms on elliptic fibrations}.
Adv. Math., 255 (2014), 86--118.

\bibitem[Lo2]{Lo11}
Lo, J.,
\emph{t-structures on elliptic fibrations}.
To appear in Kyoto J. Math.  arXiv:1509.03216 [math.AG]


\bibitem[LM]{LM}
Lo, J. and More, Y.,
\emph{Some examples of tilt-stable objects on threefolds}.
Preprint, arXiv:1209.2749 [math.AG], (2012).


\bibitem[LQ]{LQ}
Lo, J. and Qin, Z.,
\emph{Mini-walls for Bridgeland stability conditions on the derived category of sheaves over surfaces}.
Asian J. Math.,  18 (2014),  321--344.



\bibitem[MP]{MP}
Maciocia, A. and Piyaratne, D.,
\emph{Fourier-Mukai Transforms and Bridgeland Stability Conditions on Abelian Threefolds}.
Algebr. Geom., Vol. 2 (3), pp. 270-297, 2015.

\bibitem[MP2]{MP2}
Maciocia, A. and Piyaratne, D.,
\emph{Fourier-Mukai Transforms and Bridgeland Stability Conditions on Abelian Threefolds II}.
Preprint, arXiv:1310.0299 [math.AG], (2013).

\bibitem[Pol]{Pol}
A. ~Polishchuk,
\emph{Constant families of t-structures on derived categories of coherent sheaves}.
Mosc. Math. J., Vol. 7, pp. 109-134, 2007.

\bibitem[Rud07]{Rud}
Rudakov, A.,
\emph{Stability for an Abelian Category}.
J. Algebra, 197 (1997), 231--245.

\bibitem[Tod11]{Toda11-1}
Toda, Y.,
\emph{Stability conditions and curve counting invariants on Calabi–Yau 3-folds}.
Kyoto J. Math., 52 (2012), 1--50.

\bibitem[Ver07]{Ver07}
Vermeire, P.,
\emph{Moduli of reflexive sheaves on smooth projective 3-folds}.
J. Pure Appl. Algebra,  211 (2007), 622--632.

\bibitem[Yos1]{YosAS}
Yoshioka, K.,
\emph{Moduli spaces of stable sheaves on Abelian surfaces}.
Math. Ann., 321 (2001), 817--884.

\bibitem[Yos2]{YosPI}
Yoshioka, K.,
\emph{Perverse coherent sheaves and Fourier-Mukai transforms on Surfaces}.
Kyoto J. Math., 53 (2013), 261-344.


\bibitem[Yos3]{YosPII}
Yoshioka, K.,
\emph{Perverse coherent sheaves and Fourier-Mukai transforms on Surfaces II}.
Kyoto J. Math., 55 (2015), 365--459.


\end{thebibliography}
\end{document}